\documentclass[12pt]{article}
\usepackage{amsmath,amsfonts,latexsym,mathrsfs, amsthm,amssymb}
\makeatletter
  
  \@addtoreset{equation}{section}
\makeatother
\newtheorem{proposition}{Proposition}[section]
\newtheorem{theorem}{Theorem}[section]
\newtheorem{lemma}{Lemma}[section]
\newtheorem{remark}{Remark}[section]
\newtheorem{corollary}{Corollary}[section]
\newtheorem{definition}{Definition}[section]
\newtheorem{example}{Example}[section]
\usepackage{color}
\usepackage{ulem}
\date{ }

\begin{document}

\title{Stochastic optimal transport for the Langevin dynamics and its zero--mass limit
\thanks{To appear in Ann. Appl. Prob.}}

\author{
Toshio Mikami\thanks{Partially supported by JSPS KAKENHI Grant Number 24K06765.
}\\
Department of Mathematics, Tsuda University
}

\maketitle

\begin{abstract}
We introduce a stochastic optimal transport for the Langevin dynamics with positive mass and study its zero--mass limit.
The new aspect of this paper is that we only fix the initial and terminal probability distributions of the positions of particles under consideration, but not those of their velocities, with Heisenberg's uncertainty principle in mind.
In the zero--mass limit, we show that the minimizers of our stochastic optimal transport is tight if and only if the initial momentum (=(mass)$\times$(velocity)) of a particle converges to zero.
We also prove that the limit of  minimizers of our stochastic optimal transport is a minimizer of a standard stochastic optimal transport for continuous semimartingales.
\end{abstract} 



\section{Introduction}\label{sec:1}

The optimal mass transport theory plays a crucial role in many fields
e.g.,  information sciences and metric measure space
(see e.g., \cite{1-PC, 1-RR,1-V2} and the references therein).

For $d\ge 1$, let $\mathcal{P}(\mathbb{R}^d)$ denote the space of all Borel probability measures on $\mathbb{R}^d$ endowed with weak topology, and let
\begin{equation}
\mu_1(dx):=\mu(dx\times \mathbb{R}^d),\quad\mu_2(dy):=\mu(\mathbb{R}^d\times dy),\quad
\mu\in \mathcal{P}(\mathbb{R}^d\times \mathbb{R}^d),\label{1.1.322}
\end{equation}
\begin{equation}\Pi(P,Q):=\{\mu\in \mathcal{P}(\mathbb{R}^d\times \mathbb{R}^d):
\mu_1=P,\mu_2=Q\},\quad P,Q\in \mathcal{P}(\mathbb{R}^d).\label{1.2.322}
\end{equation}
Let $AC([0,1];\mathbb{R}^d)$ and $P^{X}$ denote the space of all absolutely continuous functions from $[0,1]$ to $\mathbb{R}^d$ and the probability distribution of a random variable $X$ defined on a probability space, respectively.
{\color{black}Here ``AC'' stands for absolutely continuous.}
In this paper, the probability space under consideration is not fixed.
We use the same notation $P$ for different probabilities when it is not confusing.

For $P_0,P_1\in \mathcal{P}(\mathbb{R}^d)$, the following is a typical example of the Monge--Kantorovich problem
that is a class of optimal mass transports:
\begin{equation}
T(P_0,P_1):=\inf\left\{E[|Y-X|^2]:P^X=P_0, P^Y=P_1\right\}.
\end{equation}
The following holds:
\begin{eqnarray}T(P_0,P_1)&=&\inf\left\{E\left[\int_0^1\left|\frac{d}{dt}X(t)\right|^2dt\right]:X(\cdot)\in AC([0,1];\mathbb{R}^d), {\rm a.s. }, \right.\label{0.1}\\
&&\qquad \left.P^{(X(0), X(1))}\in \Pi(P_0,P_1)\right\}.\nonumber
\end{eqnarray}
Indeed, for $\varphi (\cdot)\in AC([0,1];\mathbb{R}^d)$, by Jensen's inequality,
$$\int_0^1\left|\frac{d}{dt}\varphi (t)\right|^2dt\ge |\varphi (1)-\varphi (0)|^2,$$
where the equality holds if and only if 
$$\frac{d}{dt}\varphi (t)=\varphi (1)-\varphi (0),\quad dt{\rm -a.e.}$$ 
(see \cite{M09, M2021, 1-RR, 1-V2} and references therein).

If $T(P_0,P_1)$ is finite, then a minimizer of (\ref{0.1}) exists and satisfies the following:
\begin{equation}\label{1.2.0}
X(t)=X(0)+t\{X(1)-X(0)\},\quad 0\le t\le 1, \quad {\rm a.s.}.
\end{equation}
If, in addition, $P_0(dx)\ll dx$, then
there exists a convex function $\varphi$ on $\mathbb{R}^d$ such that 
\begin{equation}\label{1.3.0810}
X(1)=\nabla\varphi(X(0)),\quad {\rm a.s.},
\end{equation}
where $\nabla:=(\partial /\partial x_i)_{i=1}^d$.
In particular, $\{X(t)\}_{0\le t\le 1}$ is a measurable function of $(t,X(0))$
(see \cite{Bre1, Bre2,GangboMcCann,1-RR, 1-V2} and references therein).

The so-called Schr\"odinger's problem in the theory of stochastic processes 
was discussed in E. Schr\"odinger's papers \cite{S1, S2}.
It is the origin of stochastic optimal transport (SOT for short) (see \cite{M2021} and the references therein) and 
is also considered an entropic regularized optimal mass transport  in data science
nowadays (see \cite{1-PC} and the references therein).

Let
\begin{equation}\label{1.7.0325}
g(t,x):=\frac{1}{\sqrt{2\pi t}^d}\exp \left( -\frac{|x|^2}{2t}\right),\quad (t, x)\in (0,\infty )\times \mathbb{R}^d.
\end{equation}
For Borel probability measures $\mu,\nu$ on a topological space $S$, let 
\begin{equation}\label{1.8.0325}
H(\mu\|\nu):=
\begin{cases}
\displaystyle \int_S \log\frac{d\mu}{d\nu}(x)d\mu (x),&\mu\ll\nu,\\
\infty, &\hbox{otherwise.}
\end{cases}
\end{equation}

For $P_0,P_1\in \mathcal{P}(\mathbb{R}^d)$, the following is a typical example of Schr\"odinger's problem:
\begin{eqnarray}\label{1.3.-1}
&\inf\left\{H(\mu(dx\hbox{ }dy)\|P_0(dx)g(1,y-x)dy):\mu\in \Pi(P_0,P_1)\right\}
\end{eqnarray}
{\color{black}(see (\ref{1.2.322}) for notation).}
If there exists $\mu\in \Pi(P_0,P_1)$ and a Borel measurable function $q^\mu$ on  
$\mathbb{R}^d\times \mathbb{R}^d$ such that 
$$\mu (dx\hbox{ }dy)=q^\mu(x,y)P_0(dx)g(1,y-x)dy,$$
then $P_1(dy)$ has a density
$$p_1(y):=\int_{\mathbb{R}^d}q^\mu(x,y)P_0(dx)g(1,y-x).$$
In particular, (\ref{1.3.-1}) is equal to the following (see e.g., \cite{1-PC}):
\begin{eqnarray}
&&\inf\left\{E\left[\frac{|Y-X|^2}{2}\right]+H(P^{(X,Y)}\|P^{X}\times P^{Y}):
P^X=P_0, P^Y=P_1\right\}\label{1.6.0}\\
&&\qquad +\frac{d}{2}\log (2\pi )+\int_{\mathbb{R}^d}p_1(y)\log p_1(y)dy,\nonumber
\end{eqnarray} 
provided it is well defined.
The infimum in (\ref{1.6.0}) is called an entropic regularized optimal mass transport  in data science.
In information science, $H(P^{(X,Y)}\|P^{X}\times P^{Y})$
is called the mutual information of $X$ and $Y$ and is denoted by $I(X;Y)$ (see e.g., \cite{Cover}).

(\ref{1.3.-1}) is also equal to the following, which is an  SOT  analog of (\ref{0.1})
{\color{black}(see (\ref{1.2.322}), the paragraph below it, and (\ref{1.8.0325}) for notations)}:
\begin{eqnarray}
&&\inf\left\{H(P^{X(\cdot)}\|P^{X(0)+W(\cdot)}):dX(t)=u_X(t)dt +dW(t), \right.\nonumber\\
&&\qquad \left. P^{(X(0), X(1))}\in \Pi(P_0,P_1)\right\}\nonumber\\
&=&\inf\left\{E\left[\int_0^1\frac{1}{2}\left|u_X(t)\right|^2dt\right]:dX(t)=u_X(t)dt +dW(t), \right.\label{1.3.0}\\
&&\qquad \left. P^{(X(0), X(1))}\in \Pi(P_0,P_1)\right\},\nonumber
\end{eqnarray}
where $u_X(t)$ and $W(t)$ are, respectively, an $\mathbb{R}^d$--valued progressively measurable stochastic process
and an $\mathbb{R}^d$--valued  Brownian motion defined on the same filtered probability space
(see \cite{D91, F88}, \cite{J75}--\cite{MT06}, \cite{Z86}--\cite{Z86-3}
and the references therein).
 In this paper, we use the same notation $W$ for different  Brownian motions when it is not confusing.

Indeed, (\ref{1.3.-1}) is less than or equal to (\ref{1.3.0}) since if 
$$dX(t)=u_X(t)dt +dW(t),\quad dP^{X(\cdot)}\ll dP^{X(0)+W(\cdot)},$$
then 
\begin{eqnarray}
&&P^{(X(0),X(1))}(dxdy)\label{1.8.0}\\
&=&E\left[\frac{dP^{X(\cdot)}}{dP^{X(0)+W(\cdot)}}\biggl|(X(0), X(0)+W(1))=(x,y)\right]
P^{(X(0), X(0)+W(1))}(dxdy),\nonumber
\end{eqnarray}
and by Jensen's inequality,
\begin{eqnarray*}
&&H(P^{X(\cdot)}\|P^{X(0)+W(\cdot)})\\
&=&\int_{C([0,1];\mathbb{R}^d)}\left\{\frac{dP^{X(\cdot)}}{dP^{X(0)+W(\cdot)}} (\omega)\log\frac{dP^{X(\cdot)}}{dP^{X(0)+W(\cdot)}} (\omega)\right\}dP^{X(0)+W(\cdot)}(\omega)\\
&\ge& E\left[E\left[\frac{dP^{X(\cdot)}}{dP^{X(0)+W(\cdot)}}\biggl|(X(0), X(0)+W(1))\right]\right.\\
&&\qquad \times \left.\log \left(E\left[\frac{dP^{X(\cdot)}}{dP^{X(0)+W(\cdot)}}\biggl|(X(0), X(0)+W(1))\right]\right)\right]\\
&=&H(P^{(X(0),X(1))}\|P^{(X(0), X(0)+W(1))}).
\end{eqnarray*}
There exist functions $u,v$ on $\mathbb{R}^d$ such that 
$u(\cdot)+|\cdot|^2/2$ and $v(\cdot)+|\cdot|^2/2$ are convex on $\mathbb{R}^d$
and that
$$\mu_o (dx\hbox{ }dy):=\exp (-u(x)-v(y))g(1,y-x)P_0(dx)p_1(y)dy\in \Pi(P_0,P_1)$$
(see \cite{1-Jamison1974} {\color{black}and also (\ref{1.7.0325}) for notation}).
The following also has a unique weak solution:
\begin{eqnarray*}
dX_o(t)&=&\nabla_x \left\{\log \left(\int_{\mathbb{R}^d}g(1-t, y-X_o(t))\exp (-v(y))p_1(y)dy\right)\right\}dt +dW(t),\quad 0<t<1,\\
P^{X_o(0)}&=&P_0,
\end{eqnarray*}
where {\color{black}$\nabla_x:=(\partial /\partial x_i)_{i=1}^d$}, 
and the following holds:
$$dP^{X_o(\cdot)}(\omega)=\exp (-u(\omega(0))-v(\omega(0)+\omega(1)))p_1(\omega(0)+\omega(1)) dP^{X_o(0)+W(\cdot)}(\omega)$$
(see \cite{J75}).
In particular, from (\ref{1.8.0}),
\begin{eqnarray*}
P^{(X_o(0),X_o(1))}&=&\mu_o,\\
H(P^{X_o(\cdot)}\|P^{X_o(0)+W(\cdot)})
&=&H(P^{(X_o(0),X_o(1))}\|P^{(X_o(0), X_o(0)+W(1))})\\
&=&H(\mu_o\|P_0(dx)g(1,y-x)dy).
\end{eqnarray*}
If (\ref{1.3.-1}) is finite, then $\mu_o$ is  a unique minimizer of (\ref{1.3.-1})
(see \cite{1-Jamison1974, 1-RT} and also \cite{M2021}),
which implies that (\ref{1.3.-1}) is greater than or equal to (\ref{1.3.0}) and $X_o$ is a unique minimizer of (\ref{1.3.0})
(see \cite{D91, F88, J75},
\cite{M08}--\cite{MT06}, \cite{Z86}--\cite{Z86-3} and the references therein).

\begin{remark}\label{rk1.1.0}
In \cite{M04}, we showed that if $P_0(dx)\ll dx$, and $P_0$ and $P_1$ have the second moments, then 
the zero--noise limit of $X_o$ exists, and the limit is a gradient of a convex function that satisfies 
(\ref{1.3.0810}).
It is a probabilistic proof of Monge's problem
(see \cite{Bre1,Bre2,L2,M09,M2021,1-MTsiam,1-PC,1-RR,1-V2} and the references therein
and also Remark \ref{rk1.3} given later).

\end{remark}

The Langevin equation is a generalization of the Newton equation of motion and describes the motion of a particle subject to friction and stochastic forcing.

Let  $k_B$ and $T$ denote Boltzmann's constant and the absolute temperature, respectively.
For $m, \gamma >0$ and a sufficiently smooth $U:\mathbb{R}^d\rightarrow\mathbb{R}$,
the following is a class of Langevin equations with positive mass: for $t\in (0,1)$,
\begin{equation}\label{0}
m\cdot d\left(\frac{d}{dt}X(t)\right)=\left\{-\nabla U(X(t))-\gamma \left(\frac{d}{dt}X(t)\right)\right\}dt+\sqrt{2\gamma k_BT}dW(t)
\end{equation}
(see {\color{black}(\ref{1.3.0810}) for notation and also }e.g., \cite{IW14} for the SDE).
Here $m$, $U$, and $\gamma$ denote 
the mass of a particle, an interaction potential function, and
the friction coefficient, respectively.
$\gamma dX(t)/dt$ and $\sqrt{2\gamma k_BT}dW(t)/dt$ are also 
a linear dissipation and a stochastic forcing, respectively
(see e.g., \cite{PAV}).
If $\gamma=0$, then (\ref{0}) is a class of Newton equations of motion.

In (\ref{0.1}) and (\ref{1.3.0}), $X(\cdot)\in AC([0,1];\mathbb{R}^d)$ and 
$X(\cdot)$ is a semimartingale, respectively.
Inspired by the Langevin equation, we are interested in studying  an SOT in the case where 
$X(\cdot)\in AC([0,1];\mathbb{R}^d)$ and  
 $dX(t)/dt$ is a semimartingale  (see below (\ref{1.2.322}) for notation).
 
In this paper, we introduce an SOT for a class of Langevin dynamics with positive mass $m$ and study its zero--mass limit, i.e., the limit as $m\to 0$.
We consider the case where the initial and terminal probability distributions of the positions of  particles under consideration are fixed
and where the initial momentum converges to zero as $m\to 0$.
We emphasize that the initial velocity does not necessarily converge to zero as $m\to 0$.
Since we are interested in the zero--mass limit and since we assume that we know the initial and terminal probability distributions of the positions of particles,
we consider, with Heisenberg's uncertainty principle in mind, the case where we do not know the probability distributions of the momenta of particles.
Since the Langevin equation with positive mass is an SDE for the position and velocity of a particle, our problem is a new class of SOTs.
We refer readers to \cite{Birrell, FW-1, FW-2, Ishii} and the references therein for the zero--mass limit of the Langevin equation with a variable friction coefficient.


We describe our problem more precisely.
For notational simplicity, we consider the SDE for the position and momentum of a particle, instead of its velocity.
Let $\sigma:[0,1]\times \mathbb{R}^{2d}\rightarrow M(d,\mathbb{R})$
 be a bounded Borel measurable $d\times d$--matrix function.
For $m>0$, let $\mathcal{A}^m$ denote the set of $\mathbb{R}^{d}\times \mathbb{R}^{d}$--valued continuous semimartingales $\{Z(t)=(X(t),Y(t))\}_{0\le t\le 1}$ defined on a complete filtered probability space such that the following holds weakly:
\begin{eqnarray}
dX(t)&=&\frac{1}{m}Y(t)dt,\label{1.1}\\
dY(t)&=&\left\{u_X(t)-\frac{\gamma }{m}Y(t)\right\}dt+\sigma(t,Z(t))dW(t),\quad 0< t<1.\label{1.2}
\end{eqnarray}
Here  $\{u_X (t)\}_{0\le t\le 1}$ and $\{W(t)\}_{0\le t\le 1}$ are
a progressively measurable $\mathbb{R}^d$--valued stochastic process
and an $\mathbb{R}^d$--valued Brownian motion, respectively, defined on the same filtered probability space (see e.g., \cite{IW14}).
We omit the dependence of $Z\in \mathcal{A}^m$ on $m$ when it is not confusing.

\begin{remark}\label{rk1.1}
(i) For $m>0$, (\ref{1.1})--(\ref{1.2}) is equivalent to (\ref{1.1}) and the following:
\begin{eqnarray}\label{1.6}
m\cdot d\left(\frac{d}{dt}X(t)\right)&=&\left\{u_X(t)-\gamma \left(\frac{d}{dt}X(t)\right)\right\}dt\\
&&\qquad +\sigma\left(t,\left(X(t), m\frac{d}{dt}X(t)\right)\right)dW(t)\nonumber
\end{eqnarray}
{\color{black} 
(see also (\ref{3.8}) in section \ref{sec:3}).
This is the reason we use the notation $u_X(t)$ instead of 
$u_Z(t)$ in (\ref{1.1})--(\ref{1.2}).
In the stochastic control theory, one uses the notation $(u(t), Z^u(t))$ instead of 
$(u_X(t), Z(t))$ in (\ref{1.1})--(\ref{1.2})
because one defines $Z^u(t)$ from a given control $u(t)$ defined on a fixed filtered probability space
(see e.g., \cite{FS06}).
In this paper, we consider stochastic optimal control problems
where we do not fix probability space (see Definition \ref{def1} below) and  use the notation $(u_X(t), Z(t))$ instead of $(u(t), Z^u(t))$.}
\\
(ii) For $\{Z(t)=(X(t),Y(t))\}_{0\le t\le 1}\in \mathcal{A}^m$,
$Y(t)$ is $\mathcal{F}^X_t:=\sigma [X(s);0\le s\le t ]$--measurable for $t\in (0,1]$, since
$$Y(t)=
m \times \lim_{h\downarrow 0}\frac{X(t)-X(t-h)}{h},\quad 0<t\le 1, \quad {\rm a.s..}
$$
In the same way, $Y(0)$ is $\mathcal{F}^X_{0+}:=\cap_{t>0}\mathcal{F}^X_t$--measurable.
\end{remark}

For $m>0$, $B\subset \mathcal{P}(\mathbb{R}^d)$, and $P,Q\in \mathcal{P}(\mathbb{R}^d)$, let
\begin{eqnarray}
\mathcal{A}^m(B,P;Q)&:=&\{\{ Z(t)=(X(t),Y(t))\}_{0\le t\le 1}\in \mathcal{A}^m:\label{1.17.0325}\\
&&\qquad P^{(X(0), X(1))}\in \Pi(P,Q), P^{Y(0)}\in B\},\nonumber\\
\mathcal{A}^m(P;Q)&:=&\mathcal{A}^m(\mathcal{P}(\mathbb{R}^d),P_0;P_1)\label{1.18.0325}
\end{eqnarray}
{\color{black}(see around (\ref{1.2.322}) for notations).}
Let 
$$L:[0,1]\times \mathbb{R}^{2d} \times \mathbb{R}^d\longrightarrow [0,\infty )$$
be  Borel measurable.  
We also write
\begin{equation}\label{1.19.0325}
L(t,x,y;u):=L(t,z;u), \quad 
(t,u)\in [0,1]\times \mathbb{R}^d, z=(x,y)\in \mathbb{R}^d \times\mathbb{R}^d.
\end{equation}

{\color{black}We propose the following as} an SOT for the Langevin dynamics when mass $m>0$.
{\color{black}This formulation is meaningful, in that it converges, as $m\to 0$, to the SOT for the Langevin dynamics when mass $m=0$ (see Definition \ref{def1.5} and Corollaries \ref{co2.1} and \ref{co2.5} in section \ref{sec:2}).
This is the first formulation of the SOT for the Langevin dynamics with positive mass.
}

\begin{definition}[SOT for the Langevin dynamics with positive mass\label{def1}]
For $m>0$, $B\subset \mathcal{P}(\mathbb{R}^d)$, and $P,Q\in \mathcal{P}(\mathbb{R}^d)$, let
\begin{eqnarray}
V^m(B, P_0;P_1)
&:=&\inf \left\{E\biggl[\int_0^1 L(t,Z(t);u_X (t))dt \biggr]:Z\in \mathcal{A}^m(B,P_0;P_1)\right\},\label{1.3}\\
V^m(P_0,P_1)&:=&V^m(\mathcal{P}(\mathbb{R}^d), P_0;P_1).\nonumber
\end{eqnarray}
If the set over which the infimum is taken is empty, we set the infimum to be infinite.
\end{definition}

If $0<m_n\to 0, n\to\infty$, then for any sequence $\{Z_n=(X_n, Y_n)\in \mathcal{A}^{m_n}(P_0;P_1)\}_{n\ge 1}$
such that the initial momentum $Y_n(0)=m_n dX_n(t)/dt|_{t=0}$ does not converges to $0$ as $n\to\infty$, 
{\color{black}$\{P^{Z_n}\}_{n\ge 1}$
 is not  tight in $\mathcal{P}(C([0,1];\mathbb{R}^{2d}))$} since
the following is not continuous in $t$:
\begin{equation}\label{1.14}
\lim_{n\to \infty}\exp\left(-\frac{\gamma t}{m_n}\right)
=
\begin{cases}
0,&t\in (0,1],\\
1, &t=0
\end{cases}
\end{equation}
This is one of the reasons  we restrict a class of $Y(0)$ under consideration, when we study the zero--mass limit  (see  {\color{black}Remark \ref{rk2.4.03} in section \ref{sec:2.1}} and also (\ref{3.9}), (\ref{3.11}), and (\ref{3.7})).

\begin{remark}\label{rk1.3}
In (\ref{1.1})--(\ref{1.2}), if $\sigma(\cdot,\cdot)=0$,  then
$dX(\cdot)/dt\in AC([0,1];\mathbb{R}^d)$ and 
$$
m\frac{d^2}{dt^2}X(t)=u_X(t)-\gamma \frac{d}{dt}X(t).$$
In particular,  if 
$\sigma(\cdot,\cdot)=0$, then 
\begin{eqnarray}
\quad V^m(P_0,P_1)
&=&\inf \left\{E\biggl[\int_0^1 L\left(t,X(t), m\frac{d}{dt}X(t);\gamma \frac{d}{dt}X(t)+m\frac{d^2}{dt^2}X(t)\right)dt \biggr]:\right.\label{1.10}\\
&&\qquad\left.\frac{d}{dt}X(\cdot)\in AC([0,1];\mathbb{R}^d), {\rm a.s.}, 
P^{(X(0), X(1))}\in \Pi(P_0,P_1)\right\}\nonumber
\end{eqnarray}
{\color{black} (see (\ref{1.2.322}) and below it, and (\ref{1.19.0325}) for notations). }
If $\sigma(\cdot,\cdot)=0$ and $L=|u|^2$, then for $m>0$,
\begin{equation}
V^m(P_0,P_1)\le \gamma^2 T(P_0,P_1),
\end{equation}
where the equality holds if we formally substitute $m=0$ in (\ref{1.10}) (see (\ref{0.1})--(\ref{1.2.0}) and also Remark \ref{rk1.1.0}).
Indeed,
\begin{eqnarray*}
&&\int_0^1 \left|\gamma \frac{d}{dt}X(t)+m\frac{d^2}{dt^2}X(t)\right|^2dt\\
&=&\int_0^1 \left\{\gamma^2\left| \frac{d}{dt}X(t)\right|^2+m^2\left| \frac{d^2}{dt^2}X(t)\right|^2\right\}dt+
\gamma  m\left(\left| \frac{d}{dt}X(1)\right|^2-\left| \frac{d}{dt}X(0)\right|^2\right).
\end{eqnarray*}
\end{remark}

We consider the case where $m=0$.
Substitute $m=0$ in 
(\ref{1.6}).
Then we formally obtain the following SDE:
\begin{eqnarray}
\gamma dX(t)&=& u_X(t)dt+\sigma_0(t,X(t))dW(t),\quad 0< t<1,\label{1.4}
\end{eqnarray}
where
\begin{equation}\label{1.25.0325}
\sigma_0(t,x):=\sigma (t,(x,0)), \quad (t,x,0)\in [0,1]\times \mathbb{R}^d\times \mathbb{R}^d
\end{equation}
{\color{black}(see Remark \ref{rk1.1} for the reason we use the notation $u_X(t)$).}
We denote by $\mathcal{A}$ the set of $\mathbb{R}^{d}$--valued continuous semimartingales $\{X(t)\}_{0\le t\le 1}$ defined on a complete filtered probability space such that (\ref{1.4}) holds weakly.
Let 
\begin{eqnarray}
L_0(t,x;u)&:=&L(t,x,0;u), \quad (t,x,0,u)\in [0,1]\times \mathbb{R}^d \times\mathbb{R}^d \times \mathbb{R}^d,\label{1.24}\\
\mathcal{A}(P,Q)&:=&\{\{ X(t)\}_{0\le t\le 1}\in \mathcal{A}:P^{(X(0), X(1))}\in \Pi(P,Q)\},\quad P,Q\in \mathcal{P}(\mathbb{R}^d)\label{1.25}
\end{eqnarray}
{\color{black}(see around (\ref{1.2.322}) for notations).}
The following is an SOT for the Langevin dynamics with mass $m=0$,
and is a well--known class of SOTs (see e.g., \cite{M08}--\cite{MT06} and the references therein).

\begin{definition}\label{def1.5}

\noindent
For $P_0,P_1\in \mathcal{P}(\mathbb{R}^d )$, let
\begin{equation}\label{1.12}
V^0(P_0,P_1)
:=\inf \left\{E\biggl[\int_0^1 L_0(t,X(t);u_X (t))dt \biggr]:X\in \mathcal{A}(P_0,P_1)\right\}
\end{equation}
{\color{black}(see (\ref{1.24})--(\ref{1.25}) for notations).}
If the set over which the infimum is taken is empty, we set the infimum to be infinite.
\end{definition}

\begin{remark}
If $\sigma_0=0$ and $L=|u|^2$, then $V^0(P_0,P_1)=\gamma^2T(P_0,P_1)$
(see also Remarks \ref{rk1.1.0} and \ref{rk1.3}).

\end{remark}


For $m>0$ and a closed set $B\subset \mathcal{P}(\mathbb{R}^d)$, 
we show the existence of  minimizers of $V^m(B,P_0;P_1)$ when it is finite,
{\color{black}
by a standard method in the theory of the SOT.
The key idea is that  for $\{Z_n\}_{n\ge 1}\subset\mathcal{A}^m(B,P_0;P_1)$,
we prove  the tightness of  
$\{P^{Z_n}\}_{n\ge 1}$ in $\mathcal{P}(C([0,1];\mathbb{R}^{2d}))$
by that of $\{P^{Y_n(0)}\}_{n\ge 1}$ in $\mathcal{P}(\mathbb{R}^d)$ from (\ref{1.1})--(\ref{1.2}), provided (\ref{2.4}) in section \ref{sec:2.1} holds 
(see Theorem \ref{pp1.1} in section \ref{sec:2.1}).}
{\color{black}
When  $m_n\to 0, n\to\infty$, 
for $\{Z_n=(X_n, Y_n)\in \mathcal{A}^{m_n}(P_0;P_1)\}_{n\ge 1}$ 
for which (\ref{2.4}) in section \ref{sec:2.1} holds and $Y_n(0)\to 0$ as $n\to \infty$,
we also show the tightness of  
$\{P^{Z_n}\}_{n\ge 1}$ in $\mathcal{P}(C([0,1];\mathbb{R}^{2d}))$
by a standard method, from Lemmas \ref{lm3.2}--\ref{lm3.3}.}

Let $\delta_0(dy)$ denote the delta measure on $\{0\}\subset \mathbb{R}^d$
and let $D^m\subset \mathcal{P}(\mathbb{R}^d)$ ``converges'' to $\delta_0$
as $m\to 0$ (see section \ref{sec:2} for an exact meaning of ``converge''). 
{\color{black}Theorem \ref{thm2.1} in section \ref{sec:2.1} gives a lower bound of the zero--mass limit of $V^{m}(D^m,P_0;P_1), m\to 0$.
When $m_n\to 0, n\to\infty$, for $\{Z_n=(X_n, Y_n)\in \mathcal{A}^{m_n}(P_0;P_1)\}_{n\ge 1}$,
Lemma \ref{lm3.4} gives $\overline X_n\in \mathcal{A}(P_0,P^{\overline X_n(1)})$ that approximates $X_n$, which plays a central role in the proof of Theorem \ref{thm2.1} (see (\ref{1.25}) for notation).
The upper bound  of $V^{m}(D^m,P_0;P_1), m\to 0$ is a quite new result (see Theorem \ref{thm2.3} and Corollary \ref{thm2.2} in section \ref{sec:2.1}) and can be proven by Lemma \ref{lm3.6} that is very technical.
When $u\mapsto L(t,z;u)$, $(t,z)\in [0,1]\times\mathbb{R}^{2d}$ is not necessarily of polynomial growth, 
the class of $D^m$ is restrictive (see the explanation at the beginning of section \ref{sec:2.2}).
When $u\mapsto L(t,z;u)$, $(t,z)\in [0,1]\times\mathbb{R}^{2d}$  is of polynomial growth,
then we can consider a larger class of $D^m$, provided (\ref{2.52}) holds 
(see section \ref{sec:2.2}).
The convergence of minimizers can be shown in the same manner as known results (see Corollaries \ref{co2.3} and \ref{co2.5}).}
We will study the duality formula for $V^m(B,P_0;P_1)$ for $m>0$ somewhere else
(see 
section \ref{sec:2.3} for some discussion).

{\color{black}
To the best of our knowledge, this is the first paper to propose and study the SOT for the Langevin dynamics with positive mass and its zero--mass limit.

The superposition principle for the Fokker--Planck equation is closely related to the SOT and plays a crucial role in it.
It can also be considered the construction of the Nelson diffusion from a  given Fokker--Planck equation (see \cite{1-NelsonBrownian,1-NelsonQuantum} for Nelson's stochastic mechanics).
The first result was obtained by E. Carlen (see \cite{Carlen}), using a functional analytic approach.
We have been studying the problem  by way of a variational approach, which is developed as the SOT
with marginal constraints at all times (see \cite{1-MikamiCMP,1-MikamiSemi,M08,M21,M2021}, 
also \cite{1-CL1}--\cite{1-CL4} for generalizations of \cite{Carlen, 1-MikamiCMP}, and \cite{1-13, 1-62} for different approaches).
In our paper \cite{M21}, using the idea given in \cite{M08}, we 
made use of the Ambrosio--Figalli--Trevisan superposition principle (see \cite{1-Tre} and also \cite{1-Amb,1-Figalli}) to generalize our previous results on the SOT.
In D. Trevisan's paper \cite{1-Tre}, 
they only assume the $L^1$--integrability of the coefficients of differential operators in the Fokker--Planck equation. 
Since $Y(t)$ in (\ref{1.1})--(\ref{1.2}) does not necessarily have a finite expectation when 
$u\mapsto L(t,z;u)$, $(t,z)\in [0,1]\times\mathbb{R}^{2d}$ is not necessarily of polynomial growth,
we need the superposition principle generalized by V. I. Bogachev, M. R\"ockner, and S. V. Shaposhnikov (see \cite{Bog21} and (\ref{4.7.0329}) in the proof of Theorem \ref{pp1.1}).
We refer readers to \cite{1-Amb,1-AmbTre, Bernard, 1-Mikami2002ECP, Young} for  the superposition principle for the continuity equation.}

In section \ref{sec:2}, we state our results.
In section \ref{sec:3}, we give technical lemmas.
In section \ref{sec:4}, we prove our results.


{\color{black}

We describe frequently used notations for readers' convenience.

$\circ$ $\mathcal{P}(S)$ denotes the space of all Borel probability measures on a complete separable metric space $S$ endowed with weak topology.

$\circ$ $\mu_1(dx):=\mu(dx\times \mathbb{R}^d),\quad\mu_2(dy):=\mu(\mathbb{R}^d\times dy),\quad
\mu\in \mathcal{P}(\mathbb{R}^d\times \mathbb{R}^d)$ (see (\ref{1.1.322})).

$\circ$ $\Pi(P,Q):=\{\mu\in \mathcal{P}(\mathbb{R}^d\times \mathbb{R}^d):
\mu_1=P,\mu_2=Q\},\quad P,Q\in \mathcal{P}(\mathbb{R}^d)$ (see (\ref{1.2.322})). 

$\circ$ $AC([0,1];\mathbb{R}^d)$ denotes the space of all absolutely continuous functions from $[0,1]$ to $\mathbb{R}^d$ (see below (\ref{1.2.322})). 

$\circ$ $P^{X}$ denotes the probability distribution of a random variable $X$ (see below (\ref{1.2.322})). 

$\circ$ $\nabla_x:=(\partial /\partial x_i)_{i=1}^d$,
 $\partial_t:=\partial/\partial t$, $\nabla^2_x:=(\partial^2/\partial x_i\partial x_j)_{i,j=1}^d$.


$\circ$ $H(\mu\|\nu)$ denotes the relative entropy of $\mu$ with respect to $\nu$
for probability measures $\mu, \nu$ on a topological space (see (\ref{1.8.0325})). 

$\circ$ $\mathcal{A}^m:=\{\{ Z(t)=(X(t),Y(t))\}_{0\le t\le 1} \hbox{ that satisfies (\ref{1.1})--(\ref{1.2}) weakly}\}$, $m>0$.

$\circ$ $\mathcal{A}^m(B,P;Q):=\{\{ Z(t)=(X(t),Y(t))\}_{0\le t\le 1}\in \mathcal{A}^m:
P^{(X(0), X(1))}\in \Pi(P,Q),$ \\
\qquad $P^{Y(0)}\in B\}$, $m>0, B\subset \mathcal{P}(\mathbb{R}^d), P,Q\in \mathcal{P}(\mathbb{R}^d)$
 (see (\ref{1.17.0325})).

$\circ$ $\mathcal{A}^m(P;Q):=\mathcal{A}^m(\mathcal{P}(\mathbb{R}^d),P;Q)$,
$m>0, P,Q\in \mathcal{P}(\mathbb{R}^d)$
(see (\ref{1.18.0325})).

$\circ$ $L(t,x,y;u):=L(t,z;u), \quad 
(t,u)\in [0,1]\times \mathbb{R}^d, z=(x,y)\in \mathbb{R}^d \times\mathbb{R}^d$
(see (\ref{1.19.0325})).

$\circ$ $\sigma_0(t,x):=\sigma (t,(x,0)), (t,x,0)\in [0,1]\times \mathbb{R}^d\times \mathbb{R}^d$
(see (\ref{1.25.0325})).

$\circ$ $L_0(t,x;u):=L(t,x,0;u), \quad (t,x,0,u)\in [0,1]\times \mathbb{R}^d \times\mathbb{R}^d \times \mathbb{R}^d$ (see (\ref{1.24})).

$\circ$ $\mathcal{A}:=\{\{ X(t)\}_{0\le t\le 1} \hbox{ that satisfies (\ref{1.4}) weakly}\}$.

$\circ$ $\mathcal{A}(P,Q):=\{\{ X(t)\}_{0\le t\le 1}\in \mathcal{A}:P^{(X(0), X(1))}\in \Pi(P,Q)\},\quad P,Q\in \mathcal{P}(\mathbb{R}^d)$ (see (\ref{1.25})).

$\circ$ $\Psi^{m} (f)(t):=
\int_0^t (\gamma/m)\exp\left(-\gamma (t-s)/m \right)f(s)ds,\quad m>0,  t\ge 0, f:[0,t]\mapsto \mathbb{R}^d$
(see (\ref{2.9})).

$\circ$ $U_X(t):=\int_0^t u_X(s)ds$, $m>0, Z=(X,Y)\in \mathcal{A}^m, 0\le t\le1$ (see (\ref{2.7})).

$\circ$ $M_{X}(t):=\int_0^t\sigma(s,Z(s))dW(s)$, 
$m>0, Z=(X,Y)\in \mathcal{A}^m, 0\le t\le1$  (see (\ref{2.8})).

$\circ$ $K^m(t):=\gamma^{-1}\Psi^{m} (1)(t)=\gamma^{-1}
\left(1- \exp\left(-\gamma t/m\right)\right),$ $m>0, 0\le t\le 1$ (see (\ref{3.7})).

$\circ$ $
f^m(t):=\gamma K^m (1-t)=1- \exp\left(-\gamma (1-t)/m\right),$
 $m>0, 0\le t\le 1$ (see (\ref{2.12})).

$\circ$ $\varphi^m(t):=1-\int_t^1 f^m(s)^2ds,$ 
 $m>0, 0\le t\le 1$ (see (\ref{3.27})).

$\circ$ $u_X^m(t):=u_X(\varphi^m(t))$,  $m>0$, $Z=(X,Y)\in \mathcal{A}^m$, $0\le t\le 1$ (see (\ref{1.2}) and (\ref{2.26.00})).

$\circ$ $\mathcal{P}_{1,r}(\mathbb{R}^d ):=\left\{P\in \mathcal{P}(\mathbb{R}^d ):\int_{\mathbb{R}^d}|x|P(dx)\le r\right\}$, $r\ge 0$ (see (\ref{2.37.0326})).

$\circ$ $C_b(S_1;S_2)$ denotes the space of all bounded continuous functions from $S_1$ to $S_2$.

$\circ$ $C_b(S_1):=C_b(S_1;\mathbb{R})$.

$\circ$ $C_b^2(\mathbb{R}^d):=\{f\in C_b(\mathbb{R}^d)|
\partial f/\partial x_i, \partial^2f/\partial x_i\partial x_j
\in C_b(\mathbb{R}^d), i,j=1 \cdots, d\}$.

$\circ$ $C_0^2(\mathbb{R}^d):=\{f\in C_b^2(\mathbb{R}^d)|f\hbox{ has a compact support}\}$.

$\circ$ $C^{1,2}([0,1]\times \mathbb{R}^d):=\{f\in C([0,1]\times \mathbb{R}^d)|
\partial_t f, \partial f/\partial x_i, \partial^2f/\partial x_i\partial x_j\in C([0,1]\times \mathbb{R}^d), i,j=1 \cdots, d\}$.

$\circ$ $C_b^{1,2}([0,1]\times \mathbb{R}^d):=\{f\in C^{1,2}([0,1]\times \mathbb{R}^d)|
f, \partial_t f, \partial f/\partial x_i, \partial^2f/\partial x_i\partial x_j, i,j=1 \cdots, d\hbox{ are bounded}\}$.

$\circ$ $C_0^{1,2}([0,1]\times \mathbb{R}^d):=\{f\in C^{1,2}([0,1]\times \mathbb{R}^d)|f\hbox{ has a compact support}\}$.



}


\section{Main results}\label{sec:2}

In this section, we state our results.
We begin by discussing the zero--mass limit of the SOTs when $L(t,z;u)$ is not necessarily of polynomial growth in $u$ {\color{black} (see section \ref{sec:2.1})}. 
Next we address the case where $L(t,z;u)$ is of polynomial growth in $u$ 
 {\color{black} (see section \ref{sec:2.2}).
We then discuss the duality formula for the SOT  {\color{black} (see section \ref{sec:2.3})} .
We refer readers to the end of section \ref{sec:1} for notations.}

We state all assumptions before we state our results.
Let 
\begin{equation}\label{2.1.03}
C_{r,R}:=\inf\left \{\frac{L(t,z;u)}{|u|^r}: t\in [0,1], z\in \mathbb{R}^{2d},u \in \mathbb{R}^d, |u|\ge R\right\},\quad r\ge 1, R>0,
\end{equation}
\begin{equation}\label{2.2.03}
R_1(t,z;u)=R_1(L)(t,z;u):=L(t,z;u)-L(t,z;0),\quad (t,z,u)\in [0,1]\times \mathbb{R}^{2d}\times\mathbb{R}^d.
\end{equation}
For $\varepsilon_1\ge 0,\varepsilon_2\in (0,\infty]$, let
\begin{equation}\label{2.3.03}
\Delta L(\varepsilon_1,\varepsilon_2):=\sup \frac{L(t_1,z_1;u)-L(t_2,z_2;u)}{1+L(t_2,z_2;u)},
\end{equation}
where the supremum is taken over all $u\in \mathbb{R}^d $ and 
all $(t_1,z_1)$ and  $(t_2,z_2)\in [0,1]\times \mathbb{R}^{2d}$
for which $|t_1-t_2|\le\varepsilon_1$, $|z_1-z_2|< \varepsilon_2$.

We describe our assumptions {\color{black}(see Remark \ref{rk2.1} below)}.

\noindent
(A1)
(i) $\sigma:[0,1]\times\mathbb{R}^{2d}\rightarrow M(d,\mathbb{R})$ is bounded and continuous,
{\color{black}where $M(d,\mathbb{R})$ denotes the set of all $d\times d$ real matrices.}

\noindent
(ii) {\color{black}For $(t,z)\in [0,1]\times  \mathbb{R}^{2d}$, 
$\mathbb{R}^d\ni u \mapsto L(t,z;u)$ is of super--linear growth in $u$, uniformly in $t$ and $z$  (see (\ref{2.1.03}) for notation):}
\begin{equation}
\lim_{R\to\infty}C_{1,R}=\infty.
\end{equation}

\noindent
(iii)
$L:[0,1]\times\mathbb{R}^{2d}\times\mathbb{R}^d\rightarrow[0,\infty)$ is lower semicontinuous.

\noindent
(iv) For $(t,z)\in [0,1]\times\mathbb{R}^{2d}$, $L(t,z;\cdot)$ is convex.\\

\noindent
(A2) (i) $\sigma$ is a $d\times d$--identity matrix.

\noindent 
(ii) $L:[0,1]\times\mathbb{R}^{2d}\times\mathbb{R}^d
\rightarrow[0,\infty)$ is continuous.

\noindent
(iii) {\color{black}For $R_1$ defined in  (\ref{2.2.03}), the following holds:}

\begin{equation}\label{2.2}
R_1(t,z;ru)\le r^2R_1(t,z;u),\quad  t\in [0,1], z\in \mathbb{R}^{2d},u\in \mathbb{R}^d, 0<r<1.
\end{equation}
(iv) There exists $\varepsilon_0>0$ such that $\Delta L(\varepsilon_0,\infty)$ is finite
{\color{black}(see (\ref{2.3.03}) for notation)}.\\

\noindent
(A3)
(i)  {\color{black}For $u\in \mathbb{R}^d$, 
$\Delta L$ is a modulus of continuity of $[0,1]\times  \mathbb{R}^{2d}\ni (t,z)\mapsto L(t,z;u)$:}
$$\Delta L(\varepsilon_1,\varepsilon_2)\to
0, \quad \varepsilon_1,\varepsilon_2\to 0.$$

\noindent 
(ii)  For $(t,z)\in [0,1]\times\mathbb{R}^{2d}$, $\mathbb{R}^d\ni u\mapsto L(t,z;u)$ is strictly convex.

\noindent 
(iii) {\color{black}For $(t,z)\in [0,1]\times  \mathbb{R}^{2d}$, 
$\mathbb{R}^d\ni u \mapsto L(t,z;u)$ is 
of at least quadratic growth in $u$, uniformly in $t$ and $z$ (see (\ref{2.1.03}) for notation):}
\begin{equation}
\lim_{R\to\infty}C_{2,R}>0.
\end{equation}

\noindent
(A4) 
(i) There exists $C>0$ and $r_0\ge 1$ such that the following holds:
\begin{equation}\label{2.5.0}
L(t,z;u+v)\le L(t,z;u)+C|v|(|u|^{r_0-1}+|v|^{r_0-1}), \quad  t\in [0,1], z\in \mathbb{R}^{2d},u,v \in \mathbb{R}^d.
\end{equation}
(ii) {\color{black}
For $(t,z)\in [0,1]\times  \mathbb{R}^{2d}$, 
the growth order of $\mathbb{R}^d\ni u \mapsto L(t,z;u)$ is at least $r_0$ in (A4, i) above, uniformly in $t$ and $z$
(see (\ref{2.1.03}) for notation):}
\begin{equation}\label{2.7.0}
\lim_{R\to\infty}C_{r_0,R}>0.
\end{equation}

\begin{remark}\label{rk2.1}
(i) (A1, iv) and (A2, iii) imply that {\color{black}for $(t,z)\in [0,1]\times  \mathbb{R}^{2d}$}, 
$\nabla_uL (t,z;0)$ exists and is equal to $0$.
Indeed, from (\ref{2.2}),  for any subgradient $p_0\in \partial_u L (t,z;0)$, 
$$\langle p_0,rp_0\rangle \le L (t,z;rp_0)-L (t,z;0)\le r^2 \{L (t,z;p_0)-L (t,z;0)\},\quad 0<r<1,$$
{\color{black}where for $(t,z)\in [0,1]\times  \mathbb{R}^{2d}$, 
$$\partial_u L (t,z;u):=\{p\in  \mathbb{R}^d | L(t,z;v)-L(t,z;u)\ge \langle p, v-u\rangle, v\in\mathbb{R}^d\}, \quad u\in\mathbb{R}^d$$
}
(see e.g., \cite{1-V2} for subgradient and subdifferential).\\
(ii) The assumptions of Corollary \ref{thm2.2}
 are (A1, ii) and  (A2).
The following function satisfies (A1, ii)  and (A2, ii--iv), but not  (A1, iv): for $p\in (0,2)$, 
$$L=L(u)=|u|^2-|u|^p+1,\quad u\in  \mathbb{R}^d.$$
Indeed, $L (0)=1$ is a local maximum since $|u|^2<|u|^p$ if $0<|u|<1$.\\
(iii) (A2, ii, iv) implies that for $u\in \mathbb{R}^{d}$, $L (\cdot,\cdot;u)\in C_b ([0,1]\times\mathbb{R}^{2d})$
since 
$$L (t,z;u)\le L (t,0;u)+\triangle L(0,\infty)(1+L (t,0;u)),\quad (t, z)\in [0,1]\times \mathbb{R}^{2d}.$$
(iv) (A4) implies that $u\mapsto L(t,z;u)$ grows up of order $r_0$, as $|u|\to\infty$.
Indeed, substituting $u=0$ in (\ref{2.5.0}), the following holds:
$$
L(t,z;v)\le L(t,z;0)+C|v|^{r_0}, \quad  t\in [0,1], z\in \mathbb{R}^{2d},v \in \mathbb{R}^d.
$$
From (\ref{2.7.0}), for sufficientll large $R>0$,
$$0<C_{r_0,R}|v|^{r_0}\le L(t,z;v),\quad  t\in [0,1], z\in \mathbb{R}^{2d},
|v|\ge R.$$
$L=|u|^{r_0}$ satisfies (A4). Indeed, for $u,v \in \mathbb{R}^d$,
$$
|u+v|^{r_0}\le (|u|+|v|)^{r_0}\le  |u|^{r_0}+r_0|v|(|u|+|v|)^{r_0-1},
$$
since $[0,\infty)\ni x\mapsto (|u|+x)^{r_0}$ is convex if $r_0\ge 1$.
$$(|u|+|v|)^{r_0-1}\le 
\begin{cases}
2^{r_0-2}(|u|^{r_0-1}+|v|^{r_0-1}),&\quad r_0\ge 2,\\
|u|^{r_0-1}+|v|^{r_0-1},&\quad 1\le r_0<2,
\end{cases}
$$
since $[0,\infty)\ni x\mapsto x^{r_0-1}$ is convex if $r_0\ge 2$,
and since $[0,\infty)\ni x\mapsto (x+|v|)^{r_0-1}-x^{r_0-1}$ is nonincreasing if $1\le r_0<2$.\\
(v) Let $U:[0,1]\times \mathbb{R}^{2d}\rightarrow[0,\infty)$ {\color{black}be bounded and uniformly continuous}.
Let also $\{a_n,p_n\}_{n\ge 1}$ 
such that $a_n>0$ for at least one $n\ge 1$, that $a_n\ge 0, 2\le p_n<p_{n+1},  n\ge 1$, and such that 
 $$L_1(u):=\sum_{n=1}^\infty a_n |u|^{p_n}<\infty, \quad u\in\mathbb{R}^d.$$
The following is an example of $L$ that satisfies (A1, ii), (A2, ii--iv), and (A3): 
$$L (t,z;u)=L_1(u)+U(t,z),\quad (t, z,u)\in  [0,1]\times  \mathbb{R}^{2d}\times\mathbb{R}^d.$$
If there exists $n_0$ such that $a_n=0, n\ge n_0$, then
$L$ given above also satisfies  (A4).
\end{remark}

\subsection{Zero--mass limit: general cost function}\label{sec:2.1}

In this section, we do not assume that $u\mapsto L(\cdot,\cdot ;u)$ is of polynomial growth and discuss the zero--mass limit of SOTs.
{\color{black}A good aspect of the result in this section is that 
for $(t,z)\in [0,1]\times \mathbb{R}^{2d}$, $\mathbb{R}^d\ni u\mapsto L(t,z ;u)$ can be even of exponential growth.
A bad aspect is that the class of  distributions of $Y(0)$ under consideration is restrictive
(see (\ref{2.36.0}) and the discussion at the beginning of the next section).
When we consider the lower bound of the zero--mass limit, we need the convexity of $u\mapsto L(\cdot,\cdot ;u)$, but not in the case of the upper bound.

We first consider the lower bound of the zero--mass limit.}
Modifying the idea in \cite{M08,M21} (see also \cite{M2021}), the following holds
{\color{black}(see (\ref{1.17.0325}) and (\ref{1.3}) for notations)}.

\begin{theorem}\label{pp1.1}
Suppose that (A1, i, ii) holds and that $m>0$.\\
(i) Then for any closed set $B\subset \mathcal{P}(\mathbb{R}^d)$, 
any  $P_0,P_1\in\mathcal{P}(\mathbb{R}^d)$, and for any $\{Z_n=(X_n, Y_n)\}_{n\ge 1}\subset \mathcal{A}^m(B,P_0;P_1)$ such that 
\begin{equation}\label{2.4}
\sup_{n\ge 1}E\biggl[\int_0^1 L(t,Z_n(t);u_{X_n} (t))dt \biggr]<\infty,
\end{equation}
$$\{Q_n(dt\hbox{ }dz\hbox{ }du):=dtP^{(Z_n(t),u_{X_n} (t))}(dz\hbox{ }du)\}_{n\ge 1}$$ 
and $\{P^{Z_n}\}_{n\ge 1}$ are tight {\color{black}in $\mathcal{P}([0,1]\times \mathbb{R}^{2d}\times \mathbb{R}^d)$ and $\mathcal{P}(C([0,1];\mathbb{R}^{2d}))$, respectively}.
For any weak limit point $Q_\infty(dt\hbox{ }dz\hbox{ }du)$ of $\{Q_n(dt\hbox{ }dz\hbox{ }du)\}_{n\ge 1}$,
there exists $Z=(X,Y)\in \mathcal{A}^m(B,P_0;P_1)$ such that
\begin{eqnarray}
u_X(t)&=&E^{Q_\infty}[u|t,Z(t)],\quad dtdP{\rm -a.e.},\label{2.5}\\
Q_\infty(dt\hbox{ }dz\times \mathbb{R}^d)&=&dtP^{Z(t)}(dz),\label{2.6}
\end{eqnarray}
where $E^{Q_\infty}[u|t,z]$ denotes a conditional expectation of $u$ given $(t,z)$ under $Q_\infty$.\\
(ii) Suppose, in addition,  that (A1, iii, iv) holds. Then
there exist a weak limit point $Q_\infty$ of $\{Q_n\}_{n\ge 1}$ and $Z=(X,Y)\in \mathcal{A}^m(B,P_0;P_1)$ such that (\ref{2.5})--(\ref{2.6}) hold and such that 
\begin{equation}
\liminf_{n\to\infty}E\biggl[\int_0^1 L(t,Z_n(t);u_{X_n} (t))dt \biggr]
\ge E\biggl[\int_0^1 L(t,Z(t);u_{X} (t))dt \biggr].\label{1.9.1}
\end{equation}
In particular, $V^m(B,P_0;P_1)$ has a minimizer  $Z\in \mathcal{A}^m(B,P_0;P_1)$ such that (\ref{2.5})--(\ref{2.6}) hold 
for some $Q_\infty\in \mathcal{P}([0,1]\times \mathbb{R}^{2d}\times \mathbb{R}^{d})$, provided it is finite.
\end{theorem}

For $m>0$ and $f:
[0,1]\rightarrow (\mathbb{R}\cup\{\infty\})^d$, let 
\begin{eqnarray}\label{2.9}
\Psi^{m} (f)(t):&=&
\int_0^t \frac{\gamma}{m}\exp\left(-\frac{\gamma (t-s)}{m} \right)f(s)ds, 
\end{eqnarray}
for $t\in [0,1]$ such that the r. h. s. of (\ref{2.9}) is well defined.

For $m>0$ and $Z=(X,Y)\in \mathcal{A}^m$, let
\begin{eqnarray}
U_X(t):&=&\int_0^t u_X(s)ds,\quad 0\le t\le1,\label{2.7} 
\\
M_{X}(t):&=&\int_0^t\sigma(s,Z(s))dW(s),\quad 0\le t\le1\label{2.8}
\end{eqnarray}
 (see (\ref{1.2}) for notation and also Remark \ref{rk1.1}).

For $X\in \mathcal{A}$, we also use notations which are similar to (\ref{2.7})--(\ref{2.8}) when it is not confusing (see (\ref{1.4}) for notation).
In particular, $M_{X}$ also denotes a martingale part of $X\in \mathcal{A}$ for which $\sigma(s,Z(s))$ is replaced by $\sigma_0(s,X(s))$ in (\ref{2.8}).

For $m>0$ and $t\in (0,1]$, $\Psi^{m} $ is Lipschitz continuous on $C([0,t];\mathbb{R}^d)$ (see Lemma \ref{lm3.1}) and the  following holds from (\ref{1.1})--(\ref{1.2}) (see Lemma \ref{lm3.2}): for 
 $Z=(X,Y)\in \mathcal{A}^m$ and $t\in [0,1]$,
\begin{eqnarray}
X(t)&=&X(0)+\frac{1}{\gamma}\Psi^{m} (U_X+M_X)(t)
+\frac{1}{\gamma}\left\{1-\exp \left(-\frac{\gamma}{m}t\right)\right\}Y(0), \label{2.15.3}\\
Y(t)&=&U_X(t)+M_X(t)-\Psi^{m} (U_X+M_X)(t)+\exp \left(-\frac{\gamma}{m}t\right)Y(0).\label{2.16.3}
\end{eqnarray}
In particular, $(X(0), Y(0), U_X+M_X)\mapsto Z$  is Lipschitz continuous {\color{black}on 
$\mathbb{R}^d\times \mathbb{R}^d\times C([0,1];\mathbb{R}^d)$.
(\ref{2.15.3})--(\ref{2.16.3}) is useful when we consider the zero--mass limit.
Indeed, we do not need them when we only consider Theorem \ref{pp1.1} that can be proven by a standard method.}

Let $d_{wk}$ denote a metric that induces the topology by weak convergence in $\mathcal{P}(\mathbb{R}^d)$, e.g., Prohorov metric.
The following holds {\color{black}(see (\ref{1.17.0325}), (\ref{1.3}), (\ref{1.25}), and  (\ref{1.12}) for notations)}.

\begin{theorem}\label {thm2.1}
Suppose that (A1, i, ii) holds and that $\{m_n\}_{n\ge 1}$ is a sequence of positive real numbers 
that converges to $0$ as $n\to\infty$.
Then for any $\{B_n\}_{n\ge 1}\subset \mathcal{P}(\mathbb{R}^d)$ such that 
\begin{equation}\label{2.17.03}
\lim_{n\to\infty}\left(\sup\{d_{wk}
(\delta_0, P): P\in B_n\}\right)=0,
\end{equation}
any $P_0,P_1\in\mathcal{P}(\mathbb{R}^d)$, 
and for any $\{Z_n=(X_n, Y_n)\in \mathcal{A}^{m_n}(B_n,P_0;P_1)\}_{n\ge 1}$ for which
(\ref{2.4}) holds,
$\{Q_n(dt\hbox{ }dz\hbox{ }du):=dtP^{(Z_n(t),u_{X_n} (t))}(dz\hbox{ }du)\}_{n\ge 1}$ 
and $\{P^{Z_n}\}_{n\ge 1}$ are tight {\color{black}in $\mathcal{P}([0,1]\times \mathbb{R}^{2d}\times \mathbb{R}^d)$ and $\mathcal{P}(C([0,1];\mathbb{R}^{2d}))$, respectively}.
For any weak limit point $\overline Q_\infty(dt\hbox{ }dz\hbox{ }du)$ of $\{Q_n(dt\hbox{ }dz\hbox{ }du)\}_{n\ge 1}$,
there exists $X\in \mathcal{A}(P_0,P_1)$ such that
\begin{eqnarray}
u_X(t)&=&E^{\overline Q_\infty}[u|t,X(t)],\quad dtdP{\rm -a.e.},\label{2.13} 
\\
\overline Q_\infty(dt\hbox{ }dz\times \mathbb{R}^d)&=&dtP^{X(t)}(dx)\delta_0(dy).\label{2.14}
\end{eqnarray}
Suppose, in addition,  that (A1, iii, iv) holds. Then
\begin{equation}\label{2.15}
\liminf_{n\to \infty }V^{m_n}(B_n,P_0;P_1)\ge V^0(P_0,P_1).
\end{equation}
\end{theorem}

{\color{black}
\begin{remark}\label{rk2.4.03}
Since the limit of $t\mapsto \exp \left(-\gamma t/m\right)$ as $m\to 0$ is not continuous at $t=0$,
we need an assumption such as (\ref{2.17.03}) when we expect $\{P^{Z_n}\}_{n\ge 1}$ in Theorem \ref{thm2.1}
to be tight in $\mathcal{P}(C([0,1];\mathbb{R}^{2d}))$ (see (\ref{2.15.3})--(\ref{2.16.3})).
Indeed,  for $\{Z_n\}_{n\ge 1}$ in Theorem \ref{thm2.1}, $P^{X_n(0)}=P_0$ and 
we will prove the tightness of $\{P^{(Y_n(0),U_{X_n},M_{X_n})}\}_{n\ge 1}$ in 
$\mathcal{P}(\mathbb{R}^{d}\times C([0,1];\mathbb{R}^{d})\times C([0,1];\mathbb{R}^{d}))$ from Lemma \ref{lm3.2} by a standard method, which implies that of
the probability law of 
$$\left\{\left\{\left(X_n(t)-\frac{1}{\gamma}\left(1-\exp \left(-\frac{\gamma t}{m_n}\right)\right)Y_n(0),
Y_n(t)-\exp \left(-\frac{\gamma t}{m_n}\right)Y_n(0)\right)\right\}_{0\le t\le 1}\right\}_{n\ge 1}$$
in $\mathcal{P}(C([0,1];\mathbb{R}^{2d}))$
(see Lemma \ref{lm3.3}).
Therefore  $\{P^{Z_n}\}_{n\ge 1}$ is tight  in $\mathcal{P}(C([0,1];\mathbb{R}^{2d}))$ 
if and only if  $Y_n(0)\to 0$ as $n\to 0$ in law. 
\end{remark}
}

For $m>0$ and $t\in [0,1]$, let
\begin{eqnarray}
K^m(t):&=&\frac{1}{\gamma}\Psi^{m} (1)(t)=\frac{1}{\gamma}
\left(1- \exp\left(-\frac{\gamma t}{m}\right)\right),\label{3.7}\\
f^m(t):&=&\gamma K^m (1-t)=1- \exp\left(-\frac{\gamma (1-t)}{m}\right),\label{2.12}\\
\varphi^m(t):&=&1-\int_t^1 f^m(s)^2ds.\label{3.27}
\end{eqnarray}

For $P_0,P_1\in\mathcal{P}(\mathbb{R}^d)$ and $X\in \mathcal{A}(P_0,P_1)$, let $u_X$ and  $W$ satisfy (\ref{1.4}).
For $m>0$, let $W^m$ be a Brownian motion such that 
\begin{equation}\label{2.24.00}
W(\varphi^m(t))-W(\varphi^m(0))=\int_0^t f^m(s)dW^m(s),\quad t\in [0,1]
\end{equation}
by the martingale representation theorem (see e.g., \cite{IW14}).
Then, under (A2, i), the following holds from (\ref{2.12}): for $t\in [0,1]$,
\begin{eqnarray}\label{3.23}
X(\varphi^m(t))-X(\varphi^m(0))
&=&\int_{0}^{t}K^m(1-s)(u_X^m(s)f^m(s)ds +d W^m(s)),
\end{eqnarray}
where 
\begin{equation}\label{2.26.00}
u_X^m(t):=u_X(\varphi^m(t)),\quad 0\le t\le 1.
\end{equation}
We define $Z^m=(X^m, Y^m)\in \mathcal{A}^{m}$ by the following: for $t\in [0,1]$,
\begin{eqnarray}
X^m(t)&:=&X(0)+\int_0^t \frac{1}{m}Y^m(s)ds, \label{2.21}\\
Y^m(t)&=&\frac{X(\varphi^m(0))-X(0)}{K^m(1)}
+\int_0^t \left\{u_X^m(s)f^m(s)- \frac{\gamma}{m}Y^m(s)\right\}ds+W^m(t).\label{2.24}
\end{eqnarray}

{\color{black}
We consider the upper bound of the zero--mass limit (see (\ref{1.18.0325}) and (\ref{1.24})--(\ref{1.12}) for notations).
The proofs of Theorems \ref{pp1.1}--\ref{thm2.1} are modifications of the idea in \cite{M08,M21}
though we make use of (\ref{2.15.3})--(\ref{2.16.3}) in the proof of Theorem \ref{thm2.1}, which is not standard.
But that of the upper bound is not.
}
The following theorem implies that $Z^m=(X^m, Y^m)$ converges to $(X, 0)$, as $m\to 0$, 
{\color{black}locally uniformly on $[0,1)$, a.s..
Since we use Lebesgue's dominated convergence theorem in the proof,
this convergence result is enough for our proof.
The fact that $f^m(1)=0$, and $f^m(s)\to 1, m\to 0$ for $s\in [0,1)$ makes the proof 
technically complicated (see (\ref{2.12}) for notation).
Notice that (A2) does not imply the convexity of $u\mapsto L(\cdot,\cdot;u)$
(see Remark \ref{rk2.1}, (ii)).
Our proof does not imply the $\Gamma$--convergence result unless we prove the continuity of 
$$\mathcal{P}(\mathbb{R}^d)\times \mathcal{P}(\mathbb{R}^d)\ni(P,Q)\mapsto V^0(P,Q)$$
in the sense of $\Gamma$--convergence, since we make use of the fact that 
$$X^m(t)=X(t), \quad t=0,1$$
(see the proof of Theorem \ref{thm2.3}).
}

\begin{theorem}\label{thm2.3}
Suppose that (A2, i) holds and that $m>0$.
Then  for any $P_0,P_1\in\mathcal{P}(\mathbb{R}^d)$, 
any $X\in\mathcal{A}(P_0,P_1)$, and for $Z^m=(X^m, Y^m)\in \mathcal{A}^{m}$ defined by (\ref{2.21})--(\ref{2.24}), the following holds:
\begin{equation}\label{2.29.00}
X^m(1)=X(1).
\end{equation}
In particular, $Z^m=(X^m, Y^m)\in \mathcal{A}^{m}(P_0;P_1)$.
Besides, 
$(X^m, Y^m)$ converges to $(X, 0)$, as $m\to 0$, locally uniformly on $[0,1)$, a.s..\\
Suppose, in addition, that (A2, ii--iv) holds and that
\begin{equation}
E\left[\int_0^1 L_0(t,X(t);u_X(t))dt\right]<\infty.
\end{equation}
Then 
\begin{equation}\label{2.29}
\sup_{m\in (0, \varepsilon_0\gamma/2]}E\left[\int_0^1 L(t,Z^m(t);u_X^m(t)f^m(t))dt\right]<\infty,
\end{equation}
and 
\begin{equation}\label{2.30}
\limsup_{m\to 0}E\left[\int_0^1 L(t,Z^m(t);u_X^m(t)f^m(t))dt\right]\le E\left[\int_0^1 L_0(t,X(t);u_X(t))dt\right].
\end{equation}
\end{theorem}

Suppose that (A1, ii) holds.
For $m>0$ and $P_0,P_1\in \mathcal{P}(\mathbb{R}^d )$, let
\begin{equation}\label{2.3300}
C(m,P_0,P_1):=\frac{1}{f^m(0)}\inf\left\{R\cdot \varphi^m(0)+\frac{V^0(P_0,P_1)+1}{C_{1,R}
}+\{d\cdot \varphi^m(0)\}^{1/2}:R>0\right\}
\end{equation}
(see {\color{black}(\ref{1.12}) and} (\ref{2.12})--(\ref{3.27}) for notations).
If $V^0(P_0,P_1)$ is finite, then,  for $R>0$,
\begin{eqnarray}
C(m,P_0,P_1)
&\le& \frac{1}{f^m(0)}\left\{R\cdot\varphi^m(0)+\frac{V^0(P_0,P_1)+1}{C_{1,R}}+\{d\cdot\varphi^m(0)\}^{1/2}\right\}\label{2.35.0326}\\
&\to &\frac{V^0(P_0,P_1)+1}{C_{1,R}}, \quad m\to 0\nonumber\\
&\to &0, \quad R\to\infty.\nonumber
\end{eqnarray}
Besides, since $m\mapsto f^m(\cdot)$ is nonincreasing and $m\mapsto \varphi^m (\cdot)$ is nondecreasing,
\begin{equation}\label{2.36.0326}
C(m,P_0,P_1)\downarrow 0, \quad m\downarrow 0.
\end{equation}

Let 
\begin{equation}\label{2.37.0326}
\mathcal{P}_{1,r}(\mathbb{R}^d ):=\left\{P\in \mathcal{P}(\mathbb{R}^d ):\int_{\mathbb{R}^d}|x|P(dx)\le r\right\}, \quad  r\ge 0.
\end{equation}
Then $\mathcal{P}_{1,r}(\mathbb{R}^d )$ is a closed subset of $\mathcal{P}(\mathbb{R}^d )$.

The following holds from Theorem \ref{thm2.3} {\color{black}(see (\ref{1.3}) and (\ref{1.12}) for notations)}.

\begin{corollary}\label{thm2.2}
Suppose that (A1, ii) and (A2) hold. Then
for any $P_0,P_1\in \mathcal{P}(\mathbb{R}^d )$ such that $V^0(P_0,P_1)$ is finite
and any $\{D^m\}_{m\in (0, \varepsilon_0\gamma/2]}\subset \mathcal{P}(\mathbb{R}^d )$ such that 
\begin{equation}\label{2.36.0}
\mathcal{P}_{1,C(m,P_0,P_1)}(\mathbb{R}^d ) \subset D^m, \quad m\in (0, \varepsilon_0\gamma/2],
\end{equation}
$\{V^m(D^m,P_0;P_1)\}_{m\in (0, \varepsilon_0\gamma/2]}$ is bounded and the following holds: 
\begin{equation}\label{2.26}
\limsup_{m\to 0}V^m(D^m,P_0;P_1)\le V^0(P_0,P_1).
\end{equation}
\end{corollary}

\begin{remark}\label{rk2.5}
For $m>0$ and $P_0,P_1\in \mathcal{P}(\mathbb{R}^d )$,
if $D_1\subset D_2\subset \mathcal{P}(\mathbb{R}^d )$, then 
\begin{equation}
V^m(D_2,P_0;P_1)\le V^m(D_1,P_0;P_1).
\end{equation}
\end{remark}

The following holds immediately from Theorem \ref {thm2.1} and Corollary \ref{thm2.2}.
{\color{black}
Indeed, if $V^0(P_0,P_1)$ is infinite, then (\ref{2.42.0326}) holds from Theorem \ref {thm2.1} under (\ref{2.39}) (see  (\ref{2.15})).
If $V^0(P_0,P_1)$ is finite, then  (\ref{2.42.0326}) holds from Theorem \ref {thm2.1} and Corollary \ref{thm2.2} under (\ref{2.36.0}) and (\ref{2.39})
(see  (\ref{2.15}) and  (\ref{2.26})).}
We omit the proof.

\begin{corollary}\label{co2.1}
Suppose that (A1, ii, iv) and (A2) hold.
Then for any $P_0,P_1\in\mathcal{P}(\mathbb{R}^d)$ and 
any $\{D^m
\}_{m\in (0, \varepsilon_0\gamma/2]}\subset \mathcal{P}(\mathbb{R}^d)$ such that (\ref{2.36.0}) and 
\begin{equation}\label{2.39}
\sup\{d_{wk}(\delta_0, P): P\in D^m\}\to 0, \quad m\to 0,
\end{equation}
the following holds:
\begin{equation}\label{2.42.0326}
\lim_{m\to  0}V^{m}(D^m,P_0;P_1)=V^0(P_0,P_1).
\end{equation}
\end{corollary}

{\color{black}The following corollary can be proven in the same way as Lemma 3.1 in \cite{MT06}
and Proposition 2.2, (ii) in \cite{MT06} (see also \cite{M21}, Proposition 1, (iii)).
We briefly explain the proof after we state the corollary and omit the proof
(see  (\ref{1.3}) and (\ref{1.12}) for notations). 
}

\begin{corollary}\label {co2.3}
Suppose that (A1, ii, iv), (A2), and (A3, i) hold.
Then for any $P_0,P_1\in \mathcal{P}(\mathbb{R}^d )$ such that $V^0(P_0,P_1)$ is finite
and any $\{D^m\}_{m\in (0, \varepsilon_0\gamma/2]}\subset \mathcal{P}(\mathbb{R}^d)$ such that $D^m, m\in (0, \varepsilon_0\gamma/2]$ are closed and that (\ref{2.36.0}) and (\ref{2.39}) hold,
any weak limit point, as $m\to 0$, of  any minimizers of $V^{m}(D^m
, P_0;P_1)$ can be written as $(X, 0)$, where $X$ is a minimizer of $V^0(P_0,P_1)$.
Suppose, in addition, that (A3, ii, iii) holds. 
Then minimizers of $V^{m}(D^m, P_0;P_1)$ weakly converge to $(X, 0)$ as $m\to 0$,
where $X$ is the unique minimizer of $V^0(P_0,P_1)$.
\end{corollary}

{\color{black}
We explain the proof of the first part of Corollary \ref{co2.3}.
From Theorem \ref{pp1.1} and (\ref{2.42.0326}) in Corollary \ref{co2.1}, $V^{m}(D^m,P_0;P_1)$ has a minimizer $Z^m_o\in \mathcal{A}^m(D^m,P_0;P_1)$ and the following is bounded for sufficiently small $m>0$ (see (\ref{1.17.0325}) for notation):
\begin{equation}
E\left[\int_0^1L\left(t,Z^m_o(t);u_{X^m_o}(t)\right)dt\right].
\end{equation}
From Lemma \ref{lm3.3}, $P^{Z^m_o}$ has a weakly convergent subsequence, as $m\to 0$, in 
$\mathcal{P}(C([0,1];\mathbb{R}^{2d}))$ by Prokhorov's Theorem, since $C([0,1];\mathbb{R}^{2d})$ is a complete separable metric space.
The limit is the probability law of $(X,0)\in \mathcal{A}(P_0,P_1)\times \mathbb{R}^d$ from Lemmas \ref{lm3.3}--\ref{lm3.4} and Theorem 2.4 in \cite{2-JS}, p. 528
 (see  (\ref{1.25}) for notation).
The following can be proven in the same way as Lemma 3.1 in \cite{MT06}:
\begin{equation}
\lim_{m\to  0}
E\left[\int_0^1L\left(t,Z^m_o(t);u_{X^m_o}(t)\right)dt\right]
\ge E\left[\int_0^1L_0\left(t,X(t);u_{X}(t)\right)dt\right]
\end{equation}
 (see  (\ref{1.24}) for notation), since, }
for $Z=(X,Y)\in \mathcal{A}^m$ and $n\ge 2, \delta >0$, if 
$$\sup\left\{|Z(t)-Z(s)|: t,s\in [0,1], |t-s|\le \frac{1}{n}\right\}<\delta,$$
then under (A1, iv), by Jensen's inequality,
\begin{eqnarray}
&&\int_{1/n}^1
L\left(t,Z(t);n\int_{t-1/n}^t u_X(s)ds\right)dt\\
&\le& \int_{1/n}^1dt
\int_{t-1/n}^t nL\left(t,Z(t);u_X(s)\right)ds\nonumber\\
&\le& \left(1+\Delta L\left(\frac{1}{n}, \delta\right)\right)
\int_0^1 L\left(s,Z(s);u_X(s)\right)ds+\Delta L\left(\frac{1}{n}, \delta\right).\nonumber
\end{eqnarray}
The second part can also be  proven in the same way as  Proposition 2.2, (ii) in \cite{MT06}
(see also \cite{M21}, Proposition 1, (iii)).
Indeed, (A3, ii, iii) implies the uniqueness of the minimizer of $V^0(P_0,P_1)$.

\subsection{Zero--mass limit: cost function of polynomial growth in $u$}\label{sec:2.2}

In this section, we study the zero--mass limit of SOTs in the case where $u\mapsto  L(t,z;u)$ is of  polynomial growth (see Remark \ref{rk2.1}, (iv)).
{\color{black}
We can use the lower bound estimates in the previous section.
We consider the upper bound of the zero--mass limit under an additional assumption (A4)
without the convexity of $u\mapsto L(\cdot,\cdot,u)$.

As we mentioned in the previous section, (\ref{2.36.0}) is restrictive. We explain the reason below.}
For $\{m_n\}_{n\ge 1}$ such that $0<m_n\to 0, n\to\infty$ and $P_0,P_1\in\mathcal{P}(\mathbb{R}^d)$, and for $\{(X_n, Y_n)\in \mathcal{A}^{m_n}(P_0;P_1)\}_{n\ge 1}$ such that $\{E[dX_n(t)/dt|_{t=0}]\}_{n\ge 1}$ is bounded, 
$$\{P^{Y_n(0)}\}_{n\ge 1}\subset\mathcal{P}_{1,C(m_n,P_0,P_1)}(\mathbb{R}^d)$$ for sufficiently large $n\ge 1$ {\color{black}(see (\ref{2.3300}) for notation)}.
Indeed, 
{\color{black}
\begin{equation}
E[|Y_n(0)|]=m_n\times E\left[\left|\frac{dX_n(t)}{dt}\right|_{t=0}\right],
\end{equation}
and the following holds}:
\begin{equation}
\frac{C(m,P_0,P_1)}{\sqrt m}
\ge\frac{\{d\cdot\varphi^m(0)\}^{1/2}}{f^m(0)\sqrt m} 
\to \sqrt{\frac{3d}{2\gamma}},\quad m\to 0
\end{equation}
{\color{black}(see (\ref{2.12})--(\ref{3.27}) for notations).}
We show that (\ref{2.26}) holds under an additional assumption (A4),
even when $E[dX(t)/dt|_{t=0}]$ is bounded for $(X,Y)$ under consideration.

For $m>0$, $P_0,P_1\in\mathcal{P}(\mathbb{R}^d)$, and $X\in \mathcal{A}(P_0,P_1)$
{\color{black} (see (\ref{1.25}) for notation)}, 
taking a different probability space, e.g., a product probability space,  if necessary, take 
$\mathcal{Y}^m(0)$, defined on the same probability space as $\{X(t)\}_{0\le t\le 1}$,
that is independent of $X$ and 
such that as $m\to 0$,
\begin{equation}\label{2.42.0}
{\color{black}\sup_{m\in (0, \varepsilon_0\gamma/2]}E[|\mathcal{Y}^m(0)|^{r_0}]<\infty},\quad \mathcal{Y}^m(0)\to 0, {\rm a.s.},\quad E[|\mathcal{Y}^m(0)|^{r_0}]\to 0.
\end{equation}


We define $\mathcal{Z}^m=(\mathcal{X}^m, \mathcal{Y}^m)\in \mathcal{A}^{m}$ by the following: for $t\in [0,1]$,
\begin{eqnarray}
\mathcal{X}^m(t)&:=&X(0)+\int_0^t \frac{1}{m}\mathcal{Y}^m(s)ds, \label{2.42}
\\
\mathcal{Y}^m(t)&=&\mathcal{Y}^m(0)+
\int_0^t \left\{\left(u_{X}^m(s)+\beta^m(0)\right)f^m(s)- \frac{\gamma}{m}\mathcal{Y}^m(s)\right\}ds
+W^m(t),\label{2.43}
\end{eqnarray}
where
\begin{equation}\label{2.44}
\beta^m(0):=\frac{\gamma(X(\varphi^m(0))-X(0)-K^m(1)\mathcal{Y}^m(0))}{1-\varphi ^m(0)}
\end{equation}
(see (\ref{3.7})--(\ref{2.24.00}) and (\ref{2.26.00}) for notations).

The following implies that $\mathcal{Z}^m=(\mathcal{X}^m, \mathcal{Y}^m)$ 
converges to $(X, 0)$, as $m\to 0$, 
{\color{black} locally uniformly on $[0,1)$, a.s. under (A2, i).
This approximating sequence is different from that in Theorem \ref{thm2.3},
though we make use of Theorem \ref{thm2.3} in the proof.
Since we use Lebesgue's dominated convergence theorem in the proof in the same way as Theorem \ref{thm2.3},
this convergence result is also enough for our proof.
Under an additional assumption (A4), by a new approximating sequence $\mathcal{Z}^m=(\mathcal{X}^m, \mathcal{Y}^m)$, we can avoid to assume (\ref{2.36.0}) in Corollary \ref{thm2.2}.
}

\begin{theorem}\label{thm2.6}
Suppose that (A2, i) holds and that $m>0$.
Then  for any $P_0,P_1\in\mathcal{P}(\mathbb{R}^d)$, any $X\in\mathcal{A}(P_0,P_1)$, and 
any $\mathcal{Y}^m(0)$ that is independent of $X$ and such that (\ref{2.42.0}) holds, and for $\mathcal{Z}^m=(\mathcal{X}^m, \mathcal{Y}^m)\in \mathcal{A}^{m}$ defined by (\ref{2.42})--(\ref{2.43}), the following holds:
\begin{equation}
\mathcal{X}^m(1)=X(1).
\end{equation}
In particular, $\mathcal{Z}^m=(\mathcal{X}^m, \mathcal{Y}^m)\in \mathcal{A}^{m}(P_0;P_1)$
{\color{black}(see (\ref{1.18.0325}) for notation).}
Besides,
$(\mathcal{X}^m,\mathcal{Y}^m)$ converges to $(X,0)$, as $m\to 0$, locally uniformly on $[0,1)$, a.s..\\
Suppose, in addition, that (A2, ii--iv) and (A4) hold and that
\begin{equation}
E\left[\int_0^1 L_0(t,X(t);u_X(t))dt\right]<\infty
\end{equation}
{\color{black}(see (\ref{1.24}) for notation)}.
Then 
\begin{equation}\label{2.48}
\sup_{m\in (0, \varepsilon_0\gamma/2]}E\left[\int_0^1 L(t,\mathcal{Z}^m(t);(u_X^m(t)+\beta^m(0))f^m(t))dt\right]<\infty,
\end{equation}
and 
\begin{equation}\label{2.49}
\limsup_{m\to 0}E\left[\int_0^1 L(t,\mathcal{Z}^m(t);(u_X^m(t)+\beta^m(0))f^m(t))dt\right]\le E\left[\int_0^1 L_0(t,X(t);u_X(t))dt\right].
\end{equation}
\end{theorem}

The following easily holds from Theorem \ref{thm2.6}.
Indeed, if $\{Y_n\}_{n\ge 1}$ is $L^{r_0}$--convergent, then, taking a different probability space if necessary,
one can assume that $\{Y_n\}_{n\ge 1}$ is convergent a.s. by Skorokhod's theorem.
We omit the proof {\color{black}(see (\ref{1.3}) and (\ref{1.12}) for notations)}.

\begin{corollary}\label{co2.4}
Suppose that (A2) and (A4)  hold. Then
for any $P_0,P_1\in \mathcal{P}(\mathbb{R}^d )$ such that $V^0(P_0,P_1)$ is finite
and any $\{D^m \}_{m\in (0, \varepsilon_0\gamma/2]}\subset \mathcal{P}(\mathbb{R}^d)$ such that 
\begin{equation}\label{2.52}
{\color{black}\sup_{m\in (0, \varepsilon_0\gamma/2]}(\inf\{E[|Y|^{r_0}]:P^Y\in D^m\})<\infty,}\quad 
\inf\{E[|Y|^{r_0}]:P^Y\in D^m\}\to 0, \quad m\to 0,
\end{equation}
$\{V^m(D^m,P_0;P_1)\}_{m\in (0, \varepsilon_0\gamma/2]}$ is bounded and the following holds: 
\begin{equation}\label{2.50}
\limsup_{m\to 0}V^m(D^m
,P_0;P_1)\le V^0(P_0,P_1).
\end{equation}
\end{corollary}

The following holds from Theorem \ref {thm2.1} and Corollary \ref{co2.4}.
We omit the proof.

\begin{corollary}\label {co2.5}
Suppose that (A1, iv), (A2), and (A4) with $r_0>1$ hold.
Then for any $P_0,P_1\in\mathcal{P}(\mathbb{R}^d)$ and 
any $\{D^m\}_{m\in (0, \varepsilon_0\gamma/2]}\subset \mathcal{P}(\mathbb{R}^d)$ such that 
(\ref{2.39}) and (\ref{2.52}) hold,
\begin{equation}
\lim_{n\to \infty}V^{m}(D^m,P_0;P_1)=V^0(P_0,P_1).
\end{equation}
\end{corollary}

{\color{black}The following holds in the same way as Corollary \ref{co2.3} (see the explanation below  Corollary \ref{co2.3}), and we omit the proof.}

\begin{corollary}\label {co2.6}
Suppose that (A1, iv), (A2), (A3, i), and (A4) with $r_0>1$ hold.
Then for any $P_0,P_1\in \mathcal{P}(\mathbb{R}^d )$ such that $V^0(P_0,P_1)$ is finite and 
any $\{D^m\}_{m\in (0, \varepsilon_0\gamma/2]}\subset \mathcal{P}(\mathbb{R}^d)$ 
such that $D^m, m\in (0, \varepsilon_0\gamma/2]$ are closed and
 that (\ref{2.39}) and (\ref{2.52}) hold,
any weak limit point, as $m\to 0$, of
any minimizers of $V^m(D^m,P_0;P_1)$ can be written as $(X, 0)$, where $X$ is a minimizer of
$V^0(P_0,P_1)$. 
Suppose, in addition, that (A3, ii, iii) holds. Then minimizers of $V^m(D^m,P_0;P_1)$
weakly converge to $(X, 0)$ as $m\to 0$, where $X$ is the unique minimizer of $V^0(P_0,P_1)$.
\end{corollary}


\subsection{Duality formula}\label{sec:2.3}
In this section, we discuss the duality formula for $V^m(B, P_0;P_1)$.
{\color{black}We do not know a nice class of initial distributions $B\subset \mathcal{P}(\mathbb{R}^d)$ such that 
the duality formula for $V^m(B, P_0;P_1)$ is characterized by the PDE, which is our future problem.
We only explain what can be shown in the same way as the SOT for semimartingales
and why we can not set $B=\mathcal{P}(\mathbb{R}^d)$.
This section is a remark for the future research.}

By the convex duality, the following holds  in the same way as \cite{M21} (see also \cite{M2021,MT06} and the reference therein).
We give the proof for completeness  {\color{black}(see (\ref{1.3}) for notation)}.

\begin{proposition}\label{pp2.2}
Suppose that (A1) holds and that $m>0$.
Then  for any closed convex subset $B\subset \mathcal{P}(\mathbb{R}^d)$
and any $P_0,P_1\in \mathcal{P}(\mathbb{R}^d )$, 
\begin{equation}
V^m(B, P_0;P_1)
=\sup\left \{\int_{\mathbb{R}^d}f(x)P_1(dx)-V^m(B, P_0;\cdot)^*(f):f\in C_b (\mathbb{R}^d)\right\},
\end{equation}
where 
\begin{eqnarray*}
V^m(B, P_0;\cdot)^*(f)
&:=&\sup\left\{\int_{\mathbb{R}^d}f(x)P(dx)
-V^m(B,P_0;P):P\in  \mathcal{P}(\mathbb{R}^d)\right\}\\
&=&\sup \left\{E\biggl[f(X(1))-\int_0^1 L(t,Z(t);u_X (t))dt \biggr]:\right.\\
&&\qquad \left.Z\in \mathcal{A}^m, 
P^{X(0)}=P_0, P^{Y(0)}\in B\right\}.
\end{eqnarray*}
\end{proposition}

The following proposition explains why we have to consider a family $\{V^m(B, P_0;P_1), B\subset \mathcal{P}(\mathbb{R}^d)\}$ of SOTs when we consider the zero--mass limit.

\begin{proposition}\label{pp2.3}
Suppose that (A2,i) holds and that $m>0$ and $L=L(t;u)$, $(t,u)\in [0,1]\times \mathbb{R}^d$.
Then for any $P_0\in  \mathcal{P}(\mathbb{R}^d)$ and any $f\in  C_b (\mathbb{R}^d)$,
\begin{eqnarray}
&&V^m(\mathcal{P}(\mathbb{R}^d ), P_0;\cdot)^*(f)\label{2.34.1}\\
&=&\sup\left\{E\biggl[f(\eta (1))-\int_0^1 L(t;u_\eta (t))dt \biggr]:
d\eta (t) =K^m(1-t)(u_\eta(t)dt+dW(t))\right\}\notag
\end{eqnarray}
(see (\ref{3.7}) for notation), which does not depend on $P_0$,
where $u_\eta(\cdot)$ is a progressively measurable stochastic process defined on the same filtered probability space as $W(\cdot)$.
In particular, $V^m(P_0,P_1)$ does not necessarily converge to $V^0(P_0,P_1)$ as $m\to 0$
{\color{black} (see (\ref{1.3}) and  (\ref{1.12}) for notations)}.
\end{proposition}

Let 
\begin{eqnarray}
H(t,z;p)&:=&\sup\{\langle p,u\rangle -L(t,z;u):u\in\mathbb{R}^d\},\quad (t,z,p)\in [0,1]\times \mathbb{R}^{2d}\times\mathbb{R}^{d},\\
a(t,z)&:=&\sigma(t,z)\sigma(t,z)^*,\label{2.56}
\end{eqnarray}
where $\sigma(t,z)^*$ denotes the transpose of $\sigma(t,z)$.

For $f\in C_b(\mathbb{R}^d)$,
let $\psi^m(t,z)=\psi^m(t,z;f)$ be a classical solution of the following,
provided it exists (see e.g., \cite{M2021}): for $(t,z=(x,y))\in (0,1)\times \mathbb{R}^d\times\mathbb{R}^d$,
\begin{eqnarray}
0&=&\partial_t \psi^m(t,z)+\frac{1}{2}Trace(a(t,z)\nabla^2_y\psi^m(t,z))\label{2.33.0}\\
&&\qquad 
+\left\langle \frac{1}{m} \nabla_x\psi^m(t,z)-\frac{\gamma}{m}\nabla_y\psi^m(t,z), y\right\rangle+H(t,z;\nabla_y\psi^m(t,z)),\nonumber\\
\quad \psi^m(1,z)&=&f(x).\label{2.34.0}
\end{eqnarray}
{\color{black}Here $\partial_t:=\partial/\partial t,\nabla^2_y:=(\partial^2/\partial y_i\partial y_j)_{i,j=1}^d$.}

In the following example, we show that the following holds:
\begin{equation}\label{2.33.-1}
V^m(\mathcal{P}(\mathbb{R}^d ),P_0;\cdot)^*(f)=\sup\{\psi^m(0,x,y;f):y\in\mathbb{R}^d \},
\end{equation}
which does not depend on either $P_0\in \mathcal{P}(\mathbb{R}^d)$ or $x\in  \mathbb{R}^d$.

\begin{example}\label{ex2.1}
Suppose that (A2, i) holds and that $L=|u|^2/2, u\in\mathbb{R}^d$. 
For $f\in C_b^2 (\mathbb{R}^d)$ and $x\in \mathbb{R}^d$, 
let 
\begin{equation}
\phi(t,x):=
\begin{cases}
\displaystyle \log \left\{\int_{\mathbb{R}^d}\exp (f(y)) \frac{\gamma^d}{\sqrt{2\pi (1-t)}^d}
\exp \left(-\frac{\gamma^{2}|x-y|^2}{2(1-t)}\right)dy\right\},&t\in [0,1),\\
f(x),&t=1.
\end{cases}\label{2.33}
\end{equation}
Then $\phi\in C_b^{1,2}([0,1]\times \mathbb{R}^d)$, and 
$\phi(t,x)=\phi(t,x;f)$ is a classical solution of the following:
\begin{eqnarray}
\partial_t \phi(t,x) +\frac{1}{2\gamma^2}\triangle_x\phi(t,x)+\frac{1}{2\gamma^2}|\nabla_x\phi(t,x)|^2&=&0,
 \quad (t,x)\in (0,1)\times\mathbb{R}^d,\label{2.37}\\
 \phi(1,x)&=&f(x), \quad x\in \mathbb{R}^d,
\end{eqnarray}
{\color{black}where $\triangle_x:=Trace (\nabla_x^2)$.}
Let
\begin{equation}
\psi^m(t,z):=\phi(\varphi^m (t),x+K^m(1-t)y),\quad t\in [0,1], z=(x,y)\in \mathbb{R}^d\times \mathbb{R}^d\label{2.36}
\end{equation}
(see (\ref{3.7})--(\ref{3.27}) for notations).
Then  from (\ref{2.33})--(\ref{2.36}), $\psi^m(\cdot,\cdot)=\psi^m(\cdot,\cdot;f)\in C^{1,2} ([0,1]\times \mathbb{R}^{2d})$, satisfies (\ref{2.33.0})--(\ref{2.34.0}) with $a=$ an identity matrix, $H=H(p)=|p|^2/2$, and 
\begin{equation}\label{2.41}
\sup\{\psi^m(0,x,y;f):y\in\mathbb{R}^d \} 
=\sup\{\phi(\varphi^m (0),y;f):y\in\mathbb{R}^d \},\quad x\in\mathbb{R}^d.
\end{equation}

We show that (\ref{2.33.-1}) holds. 
For $Z\in \mathcal{A}^m$,  
\begin{eqnarray}\label{5.10}
&&E[f(X(1))]-E\left[\int_0^1 \frac{1}{2}|u_X(t)|^2dt\right]\\
&\le& E[\phi(\varphi^m (0),X(0)+K^m(1)Y(0);f)]\le\sup\{\phi(\varphi^m (0),y;f):y\in\mathbb{R}^d \},\nonumber
\end{eqnarray}
where the equality holds in the first inequality if $\nabla_y\psi^m(t,Z(t))=u_X(t)$.
Indeed, from (\ref{2.33.0}), by the It\^o formula, 
\begin{eqnarray*}
&&d\psi^m(t,Z(t))\\
&=&\left(\partial_t \psi^m(t,Z(t))+\frac{1}{2}\triangle_y\psi^m(t,Z(t))
+\left\langle \frac{1}{m}Y(t),\nabla_x\psi^m(t,Z(t))\right\rangle\right.\\
&&\qquad \left.+\left\langle u_X(t)-\frac{\gamma}{m}Y(t),\nabla_y\psi^m(t,Z(t))\right\rangle\right)dt\\
&&\qquad +\langle \nabla_y\psi^m(t,Z(t)), dW(t)\rangle\\
&=&\left(-\frac{1}{2}|\nabla_y\psi^m(t,Z(t))|^2+\langle u_X(t), \nabla_y\psi^m(t,Z(t))\rangle\right)dt
+\langle \nabla_y\psi^m(t,Z(t)), dW(t)\rangle\\
&=&\left(-\frac{1}{2}|\nabla_y\psi^m(t,Z(t))-u_X(t)|^2+\frac{1}{2}|u_X(t)|^2\right)dt
+\langle \nabla_y\psi^m(t,Z(t)), dW(t)\rangle.
\end{eqnarray*}

For $y\in\mathbb{R}^d$, 
 from (\ref{5.10}),
\begin{eqnarray}\label{2.67}
&&V^m(\mathcal{P}(\mathbb{R}^d ), P_0;\cdot)^*(f)\\
&\ge &\sup \left\{E\biggl[f(X(1))-\int_0^1 \frac{1}{2}|u_X(t)|^2dt \biggr]:
Z\in \mathcal{A}^m, 
P^{X(0)}=P_0, Y(0)=\frac{y-X(0)}{K^m(1)}\right\}\nonumber\\
&=&\phi(\varphi^m (0),y;f).\nonumber
\end{eqnarray}
Indeed, the following has a unique strong solution:
\begin{eqnarray*}
dX(t)&=&\frac{1}{m}Y(t)dt,\\
dY(t)&=&\left\{\nabla_y\psi^m(t,Z(t))-\frac{\gamma }{m}Y(t)\right\}dt+dW(t),\quad 0< t<1,
\end{eqnarray*}
since the following is bounded:
\begin{eqnarray*}
\nabla^2_{xy}\psi^m(t,z)&=&K^m(1-t)\nabla^2_x\phi(\varphi^m (t),x+K^m(1-t)y),\quad 
t\in [0,1], z=(x,y)\in \mathbb{R}^d\times \mathbb{R}^d,\\
\nabla^2_y\psi^m(t,z)&=&K^m(1-t)^2\nabla^2_x\phi(\varphi^m (t),x+K^m(1-t)y),\quad 
t\in [0,1], z=(x,y)\in \mathbb{R}^d\times \mathbb{R}^d,
\end{eqnarray*}
{\color{black}where 
$\nabla^2_{xy}:=(\partial^2/\partial y_i\partial x_j)_{i,j=1}^d$.}
Since $y$ is arbitrary in (\ref{2.67}), (\ref{2.41})--(\ref{2.67}) imply (\ref{2.33.-1}).
\end{example}

We will discuss a meaningful characterization of  $V^m(B,P_0;\cdot)^*(f)$ somewhere else.

\section{Lemmas}\label{sec:3}

In this section, we give technical lemmas.

The following lemma will be used in the proofs of Lemmas \ref{lm3.3}, \ref{lm3.4}, and Theorem \ref{thm2.3}.
We give the proof for completeness  (see (\ref{2.9})  for notation).

\begin{lemma}\label{lm3.1}
Let $m>0$.
(i) For $t\in [0,1]$ and $f\in C([0,t];\mathbb{R}^d)$,
\begin{equation}\label{3.1}
\Psi^{m} (f)(t)\le {\color{black}\sup_{s\in [0,t]}}|f(s)|\left(1- \exp\left (-\frac{\gamma t}{m} \right) \right).
\end{equation}
In particular, for  $f\in C([0,1];\mathbb{R}^d)$,
\begin{equation}\label{3.2}
{\color{black}\sup_{s\in [0,1]}}|\Psi^{m} (f)(s)|\le {\color{black}\sup_{s\in [0,1]}}|f(s)|.
\end{equation}
The following also holds:
\begin{equation}\label{3.3.0}
\lim_{\tilde m\to m}\left\{{\color{black}\sup_{s\in [0,1]}}|\Psi^{\tilde m} (f)(s)-\Psi^{m} (f)(s)|\right\}
=0, \quad m>0.
\end{equation}
(ii) For $t\in [0,1], \delta\in (0,1)$,  and $f\in C([0,t];\mathbb{R}^d)$, 
\begin{eqnarray}\label{3.3}
&&\left|\Psi^{m} (f)(t)-\left(1-\exp\left (-\frac{\gamma t}{m} \right)\right)f(t)\right|\\
& \le& \left\{2{\color{black}\sup_{s\in [0,t]}}|f(s)|\exp\left (-\frac{\gamma \delta}{m} \right)
+\sup_{t-t\wedge\delta\le s\le t}|f(t)-f(s)|\right\}\left(1-\exp\left (-\frac{\gamma t}{m} \right)\right),\nonumber
\end{eqnarray}
\begin{equation}\label{3.4.0}
\exp\left (-\frac{\gamma t}{m} \right)|f(t)|\le {\color{black}\sup_{s\in [0,t\wedge\delta]}}|f(s)|
+\exp\left (-\frac{\gamma \delta}{m} \right){\color{black}\sup_{s\in [0,t]}}|f(s)|,
\end{equation}
where $t\wedge\delta:=\min (t,\delta )$.
In particular, for $f\in C([0,1];\mathbb{R}^d)$ such that $f(0)=0$,
\begin{equation}\label{3.4}
{\color{black}\sup_{s\in [0,1]}}|\Psi^{m} (f)(s)-f(s)|\le 3{\color{black}\sup_{s\in [0,1]}}|f(s)|\exp\left (-\frac{\gamma \delta}{m} \right)
+2\sup_{t,s\in [0,1], |t-s|\le\delta}|f(t)-f(s)|\to 0,
\end{equation}
as $m\to 0$ and then $\delta \to 0$.
{\color{black}Here we only consider $f\in C([0,1];\mathbb{R}^d)$ such that $f(0)=0$ since $\Psi^{m} (f)(0)=0$.}

\end{lemma}

\begin{proof}
(\ref{3.1}) is true, since 
\begin{equation}\label{3.6}
\int_0^t \frac{\gamma}{m}\exp\left(-\frac{\gamma (t-s)}{m} \right)ds
=\left(1-\exp\left (-\frac{\gamma t}{m} \right)\right).
\end{equation}
 (\ref{3.3.0}) can be proven easily from the following: for {\color{black}$\tilde m$}, $m>0$ and  $s,t\in [0,1]$ such that $ s\le t$,
\begin{eqnarray*}
&&\left|\frac{\gamma}{{\color{black} \tilde m}}\exp\left(-\frac{\gamma(t-s)}{{\color{black} \tilde m}}\right)-
\frac{\gamma}{m}\exp\left(-\frac{\gamma(t-s)}{{\color{black} \tilde m}}\right)\right|
\le\gamma\left|\frac{1}{{\color{black} \tilde m}}-\frac{1}{m}\right|, \\
&&\left|\exp\left(-\frac{\gamma(t-s)}{{\color{black} \tilde m}}\right)-
\exp\left(-\frac{\gamma(t-s)}{m}\right)\right|\\
&=& \exp\left(-\frac{\gamma(t-s)}{\max({\color{black} \tilde m},m)}\right)
\left(1-\exp\left(-\left|\frac{1}{m}-\frac{1}{{\color{black} \tilde m}}\right|\gamma(t-s)\right)\right)
\le \left|\frac{1}{m}-\frac{1}{{\color{black} \tilde m}}\right|\gamma.
\end{eqnarray*}

We prove  (\ref{3.3}).
From (\ref{3.6}),
\begin{eqnarray*}
&&\Psi^{m} (f)(t)-\left(1-\exp\left (-\frac{\gamma t}{m} \right)\right)f(t)\\
&=&\left(\int_0^{t-t\wedge\delta}+ \int_{t-t\wedge\delta}^t\right)\frac{\gamma}{m}
\exp\left(-\frac{\gamma (t-s)}{m} \right)(f(s)-f(t))ds,
\end{eqnarray*}
\begin{eqnarray*}
&&\int_0^{t-t\wedge\delta}\frac{\gamma}{m}\exp\left(-\frac{\gamma (t-s)}{m} \right)|f(s)-f(t)|ds\\
&\le& 2{\color{black}\sup_{s\in [0,t]}}|f(s)|\exp\left (-\frac{\gamma (t\wedge\delta ) }{m}\right) \left(1-\exp\left (-\frac{\gamma (t-t\wedge\delta)}{m} \right)\right)\\
&&\begin{cases}
=0,&t\wedge\delta=t,\\
\le \displaystyle 2{\color{black}\sup_{s\in [0,t]}}|f(s)|\exp\left (-\frac{\gamma\delta}{m} \right)\left(1-\exp\left (-\frac{\gamma t}{m} \right)\right),&t\wedge\delta=\delta,
\end{cases}
\end{eqnarray*}
\begin{eqnarray*}
&&\int_{t-t\wedge\delta}^t\frac{\gamma}{m}\exp\left(-\frac{\gamma (t-s)}{m} \right)|f(s)-f(t)|ds\\
& \le& \sup_{s\in [t-t\wedge\delta,t]}|f(s)-f(t)|\left(1-\exp\left (-\frac{\gamma (t\wedge\delta )}{m} \right)\right).
\end{eqnarray*}
The following implies (\ref{3.4.0}):
\begin{eqnarray*}
\exp\left (-\frac{\gamma t}{m} \right)|f(t)|
\le 
\begin{cases}
|f(t\wedge\delta)|,&t\wedge\delta=t,\\
\displaystyle \exp\left (-\frac{\gamma \delta}{m} \right)|f(t)|,&t\wedge\delta=\delta.
\end{cases}
\end{eqnarray*}
 \end{proof}


The following lemma plays a crucial role in the proofs of Lemmas \ref{lm3.3}, \ref{lm3.4}, \ref{lm3.6}, and  Theorem \ref{thm2.3} (see (\ref{2.9})--(\ref{2.8}) and (\ref{3.7}) for notations).

\begin{lemma}\label{lm3.2}
For $m>0$, $Z=(X,Y)\in \mathcal{A}^m$, and $t\in[0,1]$,
\begin{eqnarray}
\quad X(t)&=&X(0)+\frac{1}{\gamma}\Psi^{m} (U_X+M_X+Y(0))(t)\label{3.8}\\
&=&X(0)+K^m(t)Y(0)+\int_0^tK^m(t-s)(u_X(s)ds+ \sigma(s,Z(s))dW(s)),\label{3.9}\\
\quad Y(t)&=&\exp \left(-\frac{\gamma}{m}t\right)Y(0)+U_X(t)+M_X(t)-\Psi^{m} (U_X+M_X)(t)\label{3.10}\\
&=&\exp \left(-\frac{\gamma}{m}t\right)Y(0)+\int_{0}^{t}\exp\left(-\frac{\gamma (t-s)}{m}\right)(u_X(s)ds +\sigma(s,Z(s))dW(s)).\label{3.11}
\end{eqnarray}
\end{lemma}

\begin{proof}
We  prove (\ref{3.8}).
Integrating (\ref{1.2}) in $t$, from (\ref{1.1}),
\begin{equation*}
\frac{d}{dt} (X(t)-X(0))=-\frac{\gamma}{m} (X(t)-X(0))+\frac{1}{m}(U_X(t)+M_X(t)+Y(0)),
\end{equation*}
which implies  (\ref{3.8}).

(\ref{3.8}) and the following imply (\ref{3.9}) (see (\ref{3.7}) for notation): by the integration by parts,
\begin{eqnarray}\label{3.11.0}
\frac{1}{\gamma}\Psi^{m} (U_X+M_X)(t)&=&\int_0^t \frac{1}{m}\exp\left(-\frac{\gamma (t-s)}{m} \right)\left(U_X(s)+M_X(s)\right)ds\\
&=&\int_0^t \left\{-\frac{d}{ds}K^m(t-s)\right\}\left(U_X(s)+M_X(s)\right)ds\nonumber\\
&=&\int_0^t K^m(t-s)\left(u_X(s)ds+ \sigma(s,Z(s))dW(s)\right),\nonumber
\end{eqnarray}
since $K^m(0)=0$ and $U_X(0)=M_X(0)=0$.

Differentiate (\ref{3.8}) and we obtain (\ref{3.10}) from (\ref{1.1}).

From (\ref{3.10}) and (\ref{3.11.0}), we obtain  (\ref{3.11}), since
$$\exp\left(-\frac{\gamma (t-s)}{m}\right)=1-\gamma K^m(t-s),\quad 0\le s\le t\le 1.$$
\end{proof}

The following lemma plays a crucial role in the proofs of Lemma \ref{lm3.4} and Theorems \ref{pp1.1}--\ref{thm2.1}
{\color{black}(see (\ref{1.17.0325}), (\ref{2.7})--(\ref{2.8}) for notations).
The tightness of $\{P^{Z_n}\}_{n\ge 1}$  in $\mathcal{P}(C([0,1];\mathbb{R}^{2d}))$ can be shown by a standard method if $\inf\{m_n|n\ge 1\}>0$,
once we show the tightness of $\{P^{Y_n(0)}\}_{n\ge 1}$ in $\mathcal{P}(\mathbb{R}^d)$.
We make use of Lemma \ref{lm3.2} and prove the tightness of $\{P^{Z_n}\}_{n\ge 1}$ in the case where  $m_n\to 0$ as $n\to\infty$.
For simplicity, we prove them in a similar manner,
instead of proving the first part by a standard method.
 }

\begin{lemma}\label{lm3.3}
Suppose that (A1, i, ii) holds and that $\{m_n\}_{n\ge 1}$ is a bounded sequence of positive numbers.
Then for any tight families $\{P_{0,n}\}_{n\ge 1}, \{P_{1,n}\}_{n\ge 1}\subset\mathcal{P}(\mathbb{R}^d)$ and any $\{Z_n=(X_n, Y_n)\in\mathcal{A}^{m_n}(P_{0,n};P_{1,n})\}_{n\ge 1}$ such that (\ref{2.4}) holds,
$\{Q_n (dt\hbox{ }dz\hbox{ }du):=dtP^{(Z_n(t),u_{X_n}(t))}(dz\hbox{ }du)\}_{n\ge 1}$ 
and {\color{black}$\{P^{(Y_n(0), U_{X_n},M_{X_n})}\}_{n\ge 1}$ are tight
in $\mathcal{P}([0,1]\times\mathbb{R}^{2d}\times\mathbb{R}^d)$ and $\mathcal{P}(\mathbb{R}^d\times 
C([0,1];\mathbb{R}^{d})\times C([0,1];\mathbb{R}^d))$, respectively.}
If $\inf\{m_n|n\ge 1\}>0$,  
then {\color{black}$\{P^{Z_n}\}_{n\ge 1}$ is tight in $\mathcal{P}(C([0,1];\mathbb{R}^{2d}))$.}
If $m_n\to 0$ as $n\to\infty$, then
{\color{black}$\{P^{Z_n}\}_{n\ge 1}$ is tight in $\mathcal{P}(C([0,1];\mathbb{R}^{2d}))$} if and only if 
$$\lim_{n\to\infty}Y_n(0)=0,\quad \hbox{in law}.$$


\end{lemma}

\begin{proof}
We first prove that $\{Q_n (dt\hbox{ }dz\hbox{ }du)\}_{n\ge 1}$ is tight {\color{black}in $\mathcal{P}([0,1]\times\mathbb{R}^{2d}\times\mathbb{R}^d)$}.
$[0,1]$ is compact.
For $k\ge 1$ and $R>0$, let 
$$B_{k,R}:=\{z\in \mathbb{R}^{k}: |z|\le R\}.$$ 
Then the following holds uniformly in $n$:
\begin{eqnarray}
Q_n ([0,1]\times B_{2d,R}^c\times \mathbb{R}^{d})&=&\int_0^1 P(|Z_n(t)|> R)dt
\to 0, \quad R\to \infty,\label{1.10.1}\\
Q_n ([0,1]\times \mathbb{R}^{2d}\times B_{d,R}^c)&=&\int_0^1 P(|u_{X_n}(t)|> R)dt\to 0, 
\quad R\to \infty.\label{1.11.1}
\end{eqnarray}

We first prove (\ref{1.10.1}). For $t\in [0,1]$,
\begin{eqnarray*}
|Z_n(t)|& \le& |X_n(t)|+|Y_n(t)|,\\
|X_n(t)|& \le& |X_n(0)|+ \frac{1}{\gamma}\left(\int_0^1|u_{X_n}(s)|ds
+\sup_{0\le \alpha\le 1}\left|\int_0^\alpha\sigma (s, Z_n(s))dW(s)\right|+|Y_n(0)|\right),\\
|Y_n(t)|& \le& |Y_n(0)|+ 2\left(\int_0^1|u_{X_n}(s)|ds
+\sup_{0\le \alpha\le 1}\left|\int_0^\alpha\sigma (s, Z_n(s))dW(s)\right|\right)
\end{eqnarray*}
from (\ref{3.2}), (\ref{3.8}), and (\ref{3.10}) {\color{black}(see (\ref{2.9}) for notation)}.

$\{P^{X_n(0)}\}_{n\ge 1}$ is tight {\color{black}in $\mathcal{P}(\mathbb{R}^d)$} since $Z_n=(X_n, Y_n)\in\mathcal{A}^{m_n}(P_{0,n},P_{1,n}), n\ge 1$.

$\{P^{Y_n(0)}\}_{n\ge 1}$ is also tight  {\color{black}in $\mathcal{P}(\mathbb{R}^d)$} from (A1, i, ii) since $Z_n=(X_n, Y_n)\in\mathcal{A}^{m_n}(P_{0,n},P_{1,n}), n\ge 1$.
Indeed, substituting $t=1$ in (\ref{3.8}),
\begin{eqnarray}\label{3.14}
&&\frac{1}{\gamma}\left(1-\exp\left(-\frac{\gamma}{m_n}\right)\right)
|Y_n(0)|\\
&\le& |X_n(1)|+|X_n(0)| +\frac{1}{\gamma}\left(1-\exp\left(-\frac{\gamma}{m_n}\right)\right)\nonumber\\
&&\qquad \times \left(\int_0^1|u_{X_n}(s)|ds
+\sup_{0\le t\le 1}\left|\int_0^t\sigma (s, Z_n(s))dW(s)\right|\right)
\nonumber
\end{eqnarray}
from (\ref{3.7}) and (\ref{3.1}).
The expectation of the third line in (\ref{3.14}) is bounded in $n$, since for  a measurable set $B\subset [0,1]$,
\begin{eqnarray}\label{4.3}
\int_B |u_{X_n}(s)|ds& \le& 
r|B|+\frac{1}{C_{1,r}}\int_0^1 L(t,Z_n(t);u_{X_n} (t))dt,\quad r>0,
\end{eqnarray}
where $|B|$ denote the Lebesgue measure  of $B$ {\color{black}(see (\ref{2.1.03}) for notation)}, and since
\begin{equation}\label{3.18}
E\left[\sup_{0\le t\le 1}\left|\int_0^t\sigma (s, Z_n(s))dW(s)\right|^2\right]
\le 4E\left[\int_0^1Trace (a (s, Z_n(s)))ds\right]
\end{equation}
by Doob's inequality (see (\ref{2.56}) for notation) .

The tightness of $\{P^{X_n(0)}\}_{n\ge 1}$ and $\{P^{Y_n(0)}\}_{n\ge 1}$  {\color{black}in $\mathcal{P}(\mathbb{R}^d)$} and  (\ref{4.3})--(\ref{3.18}) 
also imply the tightness of $\{P^{\sup_{0\le s\le 1} |X_n(s)|}\}_{n\ge 1}$ and
$\{P^{\sup_{0\le s\le 1} |Y_n(s)|}\}_{n\ge 1}$  {\color{black}in $\mathcal{P}(\mathbb{R})$} {\color{black}from (A1, i, ii)}, and  (\ref{1.10.1}) holds.

From (\ref{4.3}) {\color{black}and (A1, ii)}, (\ref{1.11.1}) can be proven by Chebychev's inequality:
$$\int_0^1 P(|u_{X_n}(t)|> R)dt
\le \int_0^1\frac{1}{R}E[|u_{X_n}(t)|]dt.
$$

From (\ref{4.3})--(\ref{3.18}), and (A1, i), 
$\{p^{( X_n(0),Y_n(0), U_{X_n},M_{X_n})}\}_{n\ge 1}$ is also tight 
{\color{black}in $\mathcal{P}(\mathbb{R}^d\times \mathbb{R}^d\times C([0,1];\mathbb{R}^{d})\times C([0,1];\mathbb{R}^d))$}
 (see \cite{2-JS}, p. 356, Theorem 4.5 and p. 363, Theorem 5.10).
 $\{m_n\}_{n\ge 1}$ is  bounded by assumption.
Take a convergent subsequence $\{m_{n_k}\}_{k\ge 1}$ and a weakly convergent subsequence 
$\{P^{(X_{n_k}(0), Y_{n_k}(0), U_{X_{n_k}},M_{X_{n_k}})}\}_{k\ge 1}$
{\color{black}in $\mathcal{P}(\mathbb{R}^d\times \mathbb{R}^d\times C([0,1];\mathbb{R}^{d})\times C([0,1];\mathbb{R}^d))$},
and denote the limit by $m$ and  $P^{(X(0), Y(0),U,M)}$, respectively{\color{black}, by Prokhorov's Theorem.
Notice that $\mathbb{R}^d$ and $C([0,1];\mathbb{R}^{d})$ with the supnorm are complete separable metric spaces.}
By Skorokhod's theorem, on a probability space, 
there exist $\mathbb{R}^d\times\mathbb{R}^d\times C([0,1];\mathbb{R}^d)\times C([0,1];\mathbb{R}^d)$--valued random variables $(\tilde X_{k}(0), \tilde Y_{k}(0), \tilde U_{k},\tilde M_{k}), k\ge 1$ and
$(\tilde X(0), \tilde Y(0),\tilde U,\tilde M)$ such that
\begin{eqnarray*}
P^{(\tilde X_{k}(0), \tilde Y_{k}(0), \tilde U_{k},\tilde M_{k})}&=&
P^{(X_{n_k}(0), Y_{n_k}(0), U_{X_{n_k}},M_{X_{n_k}})},\quad k\ge 1\\
P^{(\tilde X(0), \tilde Y(0),\tilde U,\tilde M)}&=&P^{(X(0), Y(0),U,M)},\\
(\tilde X_{k}(0), \tilde Y_{k}(0), \tilde U_{k},\tilde M_{k}) &\to& (\tilde X(0), \tilde Y(0),\tilde U,\tilde M),\quad k\to\infty, \quad {\rm a.s.}
\end{eqnarray*}
{\color{black}in $\mathbb{R}^d\times \mathbb{R}^d\times C([0,1];\mathbb{R}^{d})\times C([0,1];\mathbb{R}^d)$.}

Define $\{(\tilde X_{k},\tilde Y_{k})\}_{k\ge 1}$ by (\ref{3.8}) and (\ref{3.10})
with $(m,X(0), Y(0), U_{X},M_{X})$ replaced by $(m_k,\tilde X_{k}(0),$ $ \tilde Y_{k}(0), \tilde U_{k},\tilde M_{k})$.

If  $\inf\{m_n|n\ge 1\}>0$, then 
$m\ne 0$ and $\{(\tilde X_{k},\tilde Y_{k})\}_{k\ge 1}$ is also convergent a.s. from Lemmas \ref{lm3.1}  and \ref{lm3.2}.
Indeed, 
\begin{eqnarray*}
&&{\color{black}\sup_{t\in [0,1]}}|\Psi^{m_{n_k}}(\tilde U_{k}+\tilde M_{k})(t)
-\Psi^{m}(\tilde U+\tilde M)(t)|\\
& \le& {\color{black}\sup_{t\in [0,1]}}|\Psi^{m_{n_k}}(\tilde U_{k}+\tilde M_{k}-(\tilde U+\tilde M))(t)|
+{\color{black}\sup_{t\in [0,1]}}|\Psi^{m_{n_k}}(\tilde U+\tilde M)(t)-\Psi^{m}(\tilde U+\tilde M)(t)|
\to 0,
\end{eqnarray*}
as $k\to\infty$,
from (\ref{3.2})--(\ref{3.3.0}).

Since $\Psi^{m_{n_k}}$ is Lipschitz continuous on $C([0,1];\mathbb{R}^d)$ from Lemma \ref{lm3.1},
$P^{(\tilde X_{k},\tilde Y_{k})}$ is the same as $P^{Z_{n_k}}$ {\color{black} and $\{P^{Z_{n_k}}\}_{k\ge 1}$ is convergent in $\mathcal{P}(C([0,1];\mathbb{R}^{2d}))$.
The argument above can be applied to any subsequence of $\{Z_n\}_{n\ge 1}$.
By Prokhorov's Theorem, $\{P^{Z_n}\}_{n\ge 1}$ is tight in $\mathcal{P}(C([0,1];\mathbb{R}^{2d}))$ since
 $C([0,1];\mathbb{R}^{2d})$ with the supnorm is a complete separable metric space.}
 
 If $m_n\to 0$ as $n\to\infty$, then $m=0$ and
\begin{eqnarray}\label{3.18.1}
&&{\color{black}\sup_{t\in [0,1]}}|\Psi^{m_{n_k}}(\tilde U_{k}+\tilde M_{k})(t)-(\tilde U(t)+\tilde M(t))|\\
& \le&{\color{black}\sup_{t\in [0,1]}}|\Psi^{m_{n_k}}(\tilde U_{k}+\tilde M_{k}-(\tilde U+\tilde M))(t)|
\nonumber\\
&&\qquad +{\color{black}
\sup_{t\in [0,1]}}|\Psi^{m_{n_k}}(\tilde U+\tilde M)(t)-(\tilde U(t)+\tilde M(t))|\to 0,\nonumber
\end{eqnarray}
as $k\to\infty$, from (\ref{3.2}) and (\ref{3.4}).
{\color{black}From (\ref{1.14}), (\ref{3.8}), and (\ref{3.10}), 
$\{(\tilde X_{k},\tilde Y_{k})\}_{k\ge 1}$ is convergent in $C([0,1];\mathbb{R}^{2d})$ if and only if 
$ \tilde Y_{k}(0)\to 0$, in which case $\{P^{(\tilde X_{k},\tilde Y_{k})}=P^{Z_{n_k}}\}_{k\ge 1}$ is convergent in $\mathcal{P}(C([0,1];\mathbb{R}^{2d}))$.
The rest of the proof is the same as the case where $m\ne 0$.}

\end{proof}

For $Z=(X,Y)\in \mathcal{A}^{m}$, let
\begin{equation}\label{3.19.0}
\overline X(t):=X(0)+\frac{1}{\gamma}(U_X(t)+M_X(t)),\quad 0\le t\le 1
\end{equation}
{\color{black}(see (\ref{2.7})--(\ref{2.8}) for notations).}
The following lemma plays a crucial role  in the proofs of Theorem \ref{thm2.1} and Corollary \ref{co2.3}.

\begin{lemma}\label{lm3.4}
Suppose that (A1, i, ii) holds and that $\{m_n\}_{n\ge 1}$ is a sequence of positive real numbers 
that converges to $0$ as $n\to\infty$.
Then, for tight families $\{P_{0,n}\}_{n\ge 1}, \{P_{1,n}\}_{n\ge 1}\subset\mathcal{P}(\mathbb{R}^d)$,
 and $\{Z_n=(X_n, Y_n)\in\mathcal{A}^{m_n}(P_{0,n},P_{1,n})\}_{n\ge 1}$ such that (\ref{2.4}) holds
 and that $Y_n(0)\to 0$ as $n\to\infty$  in law,
\begin{eqnarray}
\lim_{n\to\infty}\left\{{\color{black}\sup_{t\in [0,1]}}|Y_n(t)|\right\}&=&0,\quad {\rm in\hbox{ }Prob.},\label{3.19}\\
\lim_{n\to\infty}\left\{{\color{black}\sup_{t\in [0,1]}}|X_{n}(t)-\overline X_{n}(t)|\right\}&=&0, \quad {\rm in\hbox{ }Prob.}.\label{3.20}
\end{eqnarray}
\end{lemma}

\begin{proof}
From Lemmas \ref{lm3.3}, $\{P^{(X_{n},Y_{n}, U_{X_{n}},M_{X_{n}})}\}_{n\ge 1}$ is tight
{\color{black}in $\mathcal{P}(C([0,1];\mathbb{R}^{d})\times C([0,1];\mathbb{R}^{d})\times C([0,1];\mathbb{R}^{d})\times C([0,1];\mathbb{R}^d))$}.
Take a weakly convergent subsequence
$\{P^{(X_{n_k},Y_{n_k}, U_{X_{n_k}},M_{X_{n_k}})}\}_{k\ge 1}$ {\color{black}by Prokhorov's Theorem}.
In the same way as the proof of Lemmas \ref{lm3.3},
by Skorokhod's theorem, taking a different probability space if necessary, we assume that the convergence is almost sure. 

From (\ref{3.4}) and (\ref{3.10}), 
\begin{eqnarray}\label{3.21.0}
&&{\color{black}\sup_{t\in [0,1]}}|Y_{n_k}(t)|\\
&\le&{\color{black}\sup_{t\in [0,1]}}|\Psi^{m_{n_k}}(U_{X_{n_k}}+M_{X_{n_k}})(t)-(U_{X_{n_k}}(t)+M_{X_{n_k}}(t))|+|Y_{n_k}(0)|\to 0,
\nonumber
\end{eqnarray}
as $k\to\infty$, a.s., which also implies that the following holds:
\begin{equation}\label{3.22.0}
\lim_{k\to\infty}\left\{{\color{black}\sup_{t\in [0,1]}}|X_{n_k}(t)-\overline X_{n_k}(t)|\right\}=0, \quad {\rm a.s.}.
\end{equation}
Indeed,  from (\ref{3.8}) and (\ref{3.10}),
\begin{eqnarray}\label{3.25.0}
X_{n}(t)-\overline X_{n}(t)
&=&\frac{1}{\gamma}\left\{\Psi^{m_{n}}(U_{X_{n}}+M_{X_{n}}+Y_{n}(0))(t)-(U_{X_{n}}(t)+M_{X_{n}}(t))\right\}\\
&=&\frac{1}{\gamma} (-Y_{n}(t)+Y_{n}(0)).\nonumber
\end{eqnarray}
Since (\ref{3.21.0})--(\ref{3.22.0}) holds for any weakly convergent subsequence and
since the convergence to $0$ a.s. implies that in probability, (\ref{3.19})--(\ref{3.20}) hold.
\end{proof}

The following lemma will be used in the proofs of Theorem \ref{thm2.3}
 and Corollary \ref{thm2.2}, and is given the proof for completeness
 (see (\ref{2.12})--(\ref{3.27}) for notations).

\begin{lemma}\label{lm3.5}
For $m>0$,
\begin{equation}
0\le \varphi^m(t)-t\le \frac{2m}{\gamma},\quad 0\le t\le 1.
\end{equation}
Equivalently,
\begin{equation}
0\le t-(\varphi^m)^{-1}(t)\le \frac{2m}{\gamma},\quad  \varphi^m(0)\le t\le 1.
\end{equation}
\end{lemma}

\begin{proof}
For $t\in [0,1]$,
\begin{eqnarray}\label{3.29}
\varphi^m(t)-t&=&\int_t^1 (1-f^m(s)^2)ds\ge 0.
\end{eqnarray}
The following together with (\ref{3.29}) completes the proof: for $s\in [0,1]$,
\begin{equation}
1-f^m(s)^2=2\exp\left(-\frac{\gamma (1-s)}{m}\right)-\exp\left(-\frac{2\gamma (1-s)}{m}\right)\le 2\exp\left(-\frac{\gamma (1-s)}{m}\right).
\end{equation}
\end{proof}

The following lemma plays a crucial role in the proof of Theorem \ref{thm2.3}.
\begin{lemma}\label{lm3.6}
Suppose that (A2, i) holds and that $m>0$.
Then  for any $P_0,P_1\in\mathcal{P}(\mathbb{R}^d)$, 
any $X\in\mathcal{A}(P_0,P_1)$, and for $Z^m=(X^m, Y^m)\in \mathcal{A}^{m}$ defined by (\ref{2.21})--(\ref{2.24}), the following holds:
for $t\in [0,1)$,
\begin{eqnarray}
X^m(t)&=&\left(1-\frac{K^m(t)}{K^m(1)}\right)X(0)
+f^m(t)\Psi^m\left(\frac{X(\varphi^m(\cdot))}{f^m(\cdot)^2}\right)(t),\label{3.26}\\
Y^m(t)&=&-\frac{\exp(-\gamma t/m)}{K^m(1)}X(0)
+\frac{X(\varphi^m(t))}{K^m(1-t)}
-\gamma\Psi^m\left(\frac{X(\varphi^m(\cdot))}{f^m(\cdot)^2}\right)(t)\label{2.27}
\end{eqnarray}
{\color{black}(see (\ref{1.25}) and (\ref{3.7})--(\ref{3.27}) for notations)}.

\end{lemma}

\begin{proof}

From (\ref{2.21})--(\ref{2.24}), and (\ref{3.9}), 
\begin{eqnarray}
X^m(t)&=&X(0)+K^m(t)\frac{X(\varphi^m(0))-X(0)}{K^m(1)}\label{4.13.0}\\
&&\qquad +\int_{0}^{t} K^m(t-s)(u_X^m(s)f^m(s)ds +d W^m(s)).\nonumber
\end{eqnarray}

From (\ref{4.13.0}),
 the following implies (\ref{3.26}) (see (\ref{3.7})--(\ref{2.12})
for notations): for $t\in [0,1)$,
\begin{eqnarray}\label{3.40.0}
&&\int_{0}^{t} K^m(t-s)(u_X^m(s)f^m(s)ds +d W^m(s))\\
&=&-\frac{K^m(t)}{K^m(1)}X(\varphi^m(0))
+f^m(t)\int_{0}^{t}
\frac{\gamma}{m}\exp\left(-\frac{\gamma (t-s)}{m}\right)\frac{X(\varphi^m(s))}{f^m(s)^2}ds.\nonumber
 \end{eqnarray}
We prove (\ref{3.40.0}).  
 From (\ref{3.23}), by the integration by parts,
\begin{eqnarray}
&&\int_{0}^{t} K^m(t-s)(u_X^m(s)f^m(s)ds +d W^m(s))\label{4.15.0}\\
&=&\int_{0}^{t} \frac{K^m(t-s)}{K^m(1-s)}dX(\varphi^m(s))\nonumber\\
 &=&\left[\frac{K^m(t-s)}{K^m(1-s)}X(\varphi^m(s))\right]_{s=0}^t
 -\int_{0}^{t} \frac{d}{ds}\left\{\frac{K^m(t-s)}{K^m(1-s)}\right\}X(\varphi^m(s))ds\nonumber\\
 &=&-\frac{K^m(t)}{K^m(1)}X(\varphi^m(0))
 -\int_{0}^{t} \frac{d}{ds}\left\{\frac{K^m(t-s)}{K^m(1-s)}\right\}X(\varphi^m(s))ds,\nonumber
 \end{eqnarray}
 \begin{eqnarray}\label{3.39}
&& -\frac{d}{ds}\left\{\frac{K^m(t-s)}{K^m(1-s)}\right\}\\
&=&\frac{1}{m}\left\{\frac{\exp(-\gamma (t-s)/m)K^m(1-s)-K^m(t-s)\exp(-\gamma (1-s)/m)}{K^m(1-s)^2}\right\},\nonumber
\end{eqnarray}
\begin{eqnarray}\label{3.40}
&&\exp\left(-\frac{\gamma (t-s)}{m}\right)K^m(1-s)-K^m(t-s)\exp\left(-\frac{\gamma (1-s)}{m}\right)\\
&=&\exp\left(-\frac{\gamma (t-s)}{m}\right)K^m(1-t).\nonumber
\end{eqnarray}

From (\ref{2.21})--(\ref{2.24}) and (\ref{3.11}), the following implies (\ref{2.27}) (see (\ref{3.7})--(\ref{2.12})
for notations): for $t\in [0,1)$,
\begin{eqnarray}
&&\int_{0}^{t}\exp\left(-\frac{\gamma (t-s)}{m}\right)(u_X^m(s)f^m(s)ds +d W^m(s))\label{4.23}\\
 &=&\frac{X(\varphi^m(t))}{K^m(1-t)}-\frac{\exp(-\gamma t/m)}{K^m(1)}X(\varphi^m(0))
-\gamma\Psi^m\left(\frac{X(\varphi^m (\cdot))}{f^m(\cdot)^2}\right)(t).\nonumber
\end{eqnarray}
Indeed,
\begin{eqnarray}
Y^m(t)&=&\exp\left(-\frac{\gamma t}{m}\right)\frac{X(\varphi^m(0))-X(0)}{K^m(1)}\label{4.19}\\
&&\qquad +\int_{0}^{t} \exp\left(-\frac{\gamma (t-s)}{m}\right)(u_X^m(s)f^m(s)ds +d W^m(s)).\nonumber
\end{eqnarray}

We prove (\ref{4.23}).
From   (\ref{3.23}), in the same way as (\ref{4.15.0}), by the integration by parts,
\begin{eqnarray}
&&\int_{0}^{t}\exp\left(-\frac{\gamma (t-s)}{m}\right)(u_X^m(s)f^m(s)ds +d W^m(s))\\
&=&\left[\frac{\exp(-\gamma (t-s)/m)}{K^m(1-s)}X(\varphi^m(s))\right]_{s=0}^t
 -\int_{0}^{t} \frac{d}{ds}\left\{\frac{\exp(-\gamma (t-s)/m)}{K^m(1-s)}\right\}X(\varphi^m(s))ds \nonumber\\
 &=&\frac{X(\varphi^m(t))}{K^m(1-t)}-\frac{\exp(-\gamma t/m)}{K^m(1)}X(\varphi^m(0))
-\gamma\Psi^m\left(\frac{X(\varphi^m (\cdot))}{f^m(\cdot)^2}\right)(t)\nonumber
\end{eqnarray}
from (\ref{3.39})--(\ref{3.40}).
Indeed, from (\ref{3.7}),
$$\exp\left(-\frac{\gamma (t-s)}{m}\right)=-\gamma K^m(t-s)+1,\quad 0\le s\le t\le 1,$$
$$\frac{d}{ds}\left\{\frac{1}{K^m(1-s)}\right\}=\frac{1}{mK^m(1-s)^2}
\exp\left(-\frac{\gamma (t-s)}{m}\right)\exp\left(-\frac{\gamma (1-t)}{m}\right).
$$

\end{proof}

\section{Proofs of main results}\label{sec:4}

In this section, we prove our results.

For $m>0$, $t\in [0,1], z=(x,y)
\in  \mathbb{R}^{d}\times \mathbb{R}^{d}, u\in \mathbb{R}^{d}$, and $f\in C^{1,2}([0,1]\times \mathbb{R}^{2d})$,
let 
\begin{eqnarray}
\mathcal{L}^m_{t,z,u}f(t,z)
&:=&\partial_tf(t, z)+\left\langle \frac{1}{m}y,\nabla_x f(t, z)\right\rangle+\left\langle u-\frac{\gamma}{m}y,\nabla_y f(t, z)\right\rangle\label{4.1.0420}\\
&&\qquad +\frac{1}{2}Trace (a(t,z)\nabla_y^2f(t, z))\nonumber
\end{eqnarray}
(see (\ref{2.56}) for notation), where $\nabla_y^2:=(\partial/\partial y_i\partial y_j)_{i,j=1}^d$.

Modifying the idea in \cite{M08,M21} (see also \cite{M2021} and the reference therein), 
we prove Theorem \ref{pp1.1} by the superposition principle (see \cite{Bog21}).
We refer readers to the discussion at the end of section \ref{sec:1}.

\begin{proof}[Proof of Theorem \ref{pp1.1}]
From Lemma \ref{lm3.3}, $\{Q_n\}_{n\ge 1}$ and $\{P^{Z_n}\}_{n\ge 1}$ are tight
{\color{black}in $\mathcal{P}([0,1]\times\mathbb{R}^{2d}\times\mathbb{R}^d)$ and $\mathcal{P}(C([0,1];\mathbb{R}^{2d}))$, respectively.
By Prokhrov's Theorem, }take weakly convergent subsequences $\{Q_{n_k}\}_{k\ge 1}$ and $\{P^{Z_{n_k}}\}_{k\ge 1}$ of $\{Q_n\}_{n\ge 1}$ and $\{P^{Z_n}\}_{n\ge 1}$, respectively.
Let $Q_\infty$ and $P^{Z_\infty}$ denote the limits of $Q_{n_k} $ and $\{P^{Z_{n_k}}\}_{k\ge 1}$
as $k\to\infty$, respectively.
Then 
\begin{equation}\label{4.2}
Q_\infty(dt\hbox{ }dz\times \mathbb{R}^d)=dtP^{Z_\infty(t)}(dz),
\end{equation}
since 
$$\int_{[0,1]\times \mathbb{R}^{2d}}f(t,z)Q_{n_k}(dt\hbox{ }dz\times \mathbb{R}^d)
=\int_0^1 E[f(t,Z_{n_k}(t))]dt,\quad f\in C_b([0,1]\times \mathbb{R}^{2d}).
$$

We prove (\ref{2.5})--(\ref{2.6}).
For any $f\in C_0^{2}(\mathbb{R}^{2d})$ and $t\in [0,1]$, from (\ref{4.2}),
\begin{eqnarray}\label{4.3.0}
&&\int_{ \mathbb{R}^{2d}}f(z)P^{Z_\infty(t)}(dz)-
\int_{ \mathbb{R}^{2d}}f(z)P^{Z_\infty(0)}(dz)\\
&=&\int_{[0,t]\times \mathbb{R}^{2d}}
\left\{\mathcal{L}^m_{s,z, E^{Q_\infty}[u|s,z]}f(z)\right\}dsP^{Z_\infty(s)}(dz).\nonumber
\end{eqnarray}
Indeed, for any $\varphi\in C_0^{1,2}((0,1)\times \mathbb{R}^{2d})$, by the It\^o formula,  
\begin{eqnarray}\label{4.4}
0&=&\int_{[0,1]}dt
E\left[\mathcal{L}^m_{t,Z_{n_k}(t),u_{X_{n_k}}(t)}\varphi(t,Z_{n_k}(t))\right]\\
&=&\int_{[0,1]\times \mathbb{R}^{2d}\times \mathbb{R}^{d}}
\mathcal{L}^m_{t,z,u}\varphi(t,z)Q_{n_k} (dt\hbox{ }dz\hbox{ }du)\nonumber\\
&\to&\int_{[0,1]\times \mathbb{R}^{2d}\times \mathbb{R}^{d}}
\mathcal{L}^m_{t,z,u}\varphi(t,z)Q_\infty(dt\hbox{ }dz\hbox{ }du),
\quad k\to\infty\nonumber\\
&=&\int_{[0,1]\times \mathbb{R}^{2d}}
\mathcal{L}^m_{t,z, E^{Q_\infty}[u|t,z]}\varphi(t,z)Q_\infty(dt\hbox{ }dz\times \mathbb{R}^d),\nonumber
\end{eqnarray}
from (A1, i, ii).
Here we used the following: from (\ref{2.4}), for $R>0$,
\begin{eqnarray}\label{4.5.0}
&&\int_{[0,1]\times \mathbb{R}^{2d}\times \{u\in \mathbb{R}^d:|u|>R\} } |u|Q_{n_k} (dt\hbox{ }dz\hbox{ }du)\\
& \le& \frac{1}{C_{1,R}
}\int_{[0,1]\times \mathbb{R}^{2d}\times \{u\in \mathbb{R}^d:|u|>R\} } L(t,z;u)Q_{n_k} (dt\hbox{ }dz\hbox{ }du)\nonumber\\
&\le&  \frac{1}{C_{1,R}
}
E\left[\int_0^1 L(t,Z_{n_k}(t);u_{X_{n_k}}(t))dt\right] 
\to 0,\nonumber
\end{eqnarray}
as $R\to\infty$, uniformly in $k$
{\color{black}(see (\ref{2.1.03}) for notation)}.


From (\ref{4.3.0}), the superposition principle proven in \cite{Bog21} implies that 
there exists $Z=(X,Y)\in \mathcal{A}^m$ such that 
\begin{equation}\label{4.6.0}
P^{Z(t)}(dz)=P^{Z_\infty(t)}(dz),\quad 0\le t\le 1,
\end{equation}
and that (\ref{2.5})--(\ref{2.6}) hold (see \cite{Bog21}, Theorem 1.1).
{\color{black}
We show that (1.3) in \cite{Bog21} that is the assumption in Theorem 1.1 in \cite{Bog21} holds 
under our assumption.
For the coefficients 
$$b^{u,m}(t,z):=\left(\frac{y}{m}, E^{Q_\infty}[u|t,z]-\frac{\gamma}{m}y\right)^*, \quad \frac{1}{2}a(t,z),\quad t\in [0,1], z=(x,y)\in \mathbb{R}^{d}\times \mathbb{R}^{d}$$
of $\mathcal{L}^m_{t,z, E^{Q_\infty}[u|t,z]}$, 
\begin{equation}\label{4.7.0329}
\frac{\|a(t,z)\|+|\langle b^{u,m}(t,z),z\rangle|}{(1+|z|)^2}
\le \sup_{(t,z)\in[0,1]\times\mathbb{R}^{2d}}\|a(t,z)\|+\frac{1}{m}+|E^{Q_\infty}[u|t,z]|+\frac{\gamma}{m},
\end{equation}
where $\|a(t,z)\|:=\sqrt{\sum_{i,j=1}^d a_{ij}(t,z)^2}$.
From (A1, i) and (\ref{4.5.0}), 
$b^{u,m},a\in L^1([0,1]\times B_{R},dtP^{Z_\infty(t)}(dz))$ for any $R>0$,
where
$B_{R}:=\{z\in \mathbb{R}^{2d}: |z|\le R\}$,
and the right hand side of (\ref{4.7.0329})  belongs to the set $L^1([0,1]\times \mathbb{R}^{2d},dtP^{Z_\infty(t)}(dz))$.
Indeed, from (\ref{4.5.0}), for sufficinetly large $R>0$, 
\begin{eqnarray*}
&&\int_{[0,1]\times \mathbb{R}^{2d}}
|E^{Q_\infty}[u|t,z]|Q_\infty(dt\hbox{ }dz\times \mathbb{R}^d)\nonumber\\
&\le& \int_{[0,1]\times \mathbb{R}^{2d}\times \mathbb{R}^{d}} |u|Q_{\infty} (dt\hbox{ }dz\hbox{ }du)\\
& \le& R+\liminf_{k\to\infty}\frac{1}{C_{1,R}
}\int_{[0,1]\times \mathbb{R}^{2d}\times \{u\in \mathbb{R}^d:|u|>R\}} L(t,z;u)Q_{n_k} (dt\hbox{ }dz\hbox{ }du)<\infty.
\end{eqnarray*}}
$P^{X(t)}=P_t, t=0,1$ and $P^{Y(0)}\in B $ from (\ref{4.6.0}), since $Z_{n_k}\in\mathcal{A}^m (B ,P_0;P_1)$
and since $B$ is a closed set.

We prove (\ref{1.9.1}).
Taking a subsequence if necessary, we assume that the following  is convergent:
$$E\biggl[\int_0^1 L(t,Z_n(t);u_{X_n} (t))dt \biggr].$$
Take a weakly convergent subsequence $\{Q_{n_k}\}_{k\ge 1}$ of $\{Q_{n}\}_{n\ge 1}
{\color{black}\subset \mathcal{P}([0,1]\times \mathbb{R}^{2d}\times\mathbb{R}^d)}$,
its weak limit $Q_\infty$, and $Z=(X,Y)\in \mathcal{A}^m(B,P_0;P_1)$ that satisfies (\ref{2.5})--(\ref{2.6}).
By Skorokhod's theorem, on a probability space, there exist $[0,1]\times \mathbb{R}^{2d}\times \mathbb{R}^{d}$--valued random variables $(T_k, Z_k, U_k), k\ge 1$ and
$(T, Z, U)$ such that 
$$P^{(T_k, Z_k, U_k)}=Q_{n_k},\quad k\ge 1,\quad P^{(T, Z, U)}=Q_\infty,$$
$$\lim_{k\to\infty} (T_k, Z_k, U_k)=(T, Z, U),\quad {\rm a.s.}.$$
By Fatou's lemma and Jensen's inequality, from (A1, iii, iv), the following holds:
\begin{eqnarray}\label{4.7.0}
&&\lim_{k\to\infty}E\biggl[\int_0^1 L(t,Z_{n_k}(t);u_{X_{n_k}} (t))dt \biggr]\\
&=&\lim_{k\to\infty}
\int_{[0,1]\times \mathbb{R}^{2d}\times \mathbb{R}^{d}}
L(t,z;u)Q_{n_k}(dt\hbox{ }dz\hbox{ }du)
=\lim_{k\to\infty}E[L(T_k, Z_k;U_k)]\nonumber\\
& \ge &E[L(T, Z; U)]=\int_{[0,1]\times \mathbb{R}^{2d}\times \mathbb{R}^{d}}
L(t,z;u)Q_\infty (dt\hbox{ }dz\hbox{ }du)\nonumber\\
& \ge&\int_{[0,1]\times \mathbb{R}^{2d}}
L(t,z;E^{Q_\infty}[u|t,z])Q_\infty (dt\hbox{ }dz\times \mathbb{R}^d)
=E\biggl[\int_0^1 L(t,Z(t);u_{X} (t))dt \biggr].\nonumber
\end{eqnarray}

\end{proof}

Theorem \ref{thm2.1} can be proven in the same way as Theorem \ref{pp1.1}.
The key is Lemma \ref{lm3.4} which implies (\ref{4.7}) given below.

\begin{proof}[Proof of Theorem \ref{thm2.1}]
$\{Q_n\}_{n\ge 1}$ and $\{P^{Z_n=(X_n, Y_n)}\}_{n\ge 1}$ are tight {\color{black}in $\mathcal{P}([0,1]\times \mathbb{R}^{2d}\times\mathbb{R}^d)$ and $\mathcal{P}(C([0,1]\times \mathbb{R}^{2d}))$,
respectively},
from Lemma \ref{lm3.3} and  from assumption.
{\color{black}By Prokhrov's Theorem, }take weakly convergent subsequences 
$\{Q_{n_k}\}_{k\ge 1}$ and 
$\{P^{Z_{n_k}}\}_{k\ge 1}$
of $\{Q_n \}_{n\ge 1}$ and $\{P^{Z_n}\}_{n\ge 1}$, respectively, and let $\overline Q_\infty$ denote the limit of $Q_{n_k} $ as $k\to\infty$.
Then from Lemma \ref{lm3.4},
\begin{equation}
\overline Q_\infty(dt\hbox{ }dz\hbox{ }du)=dt\delta_0(dy)\overline Q_\infty(dt\hbox{ }dx\times \mathbb{R}^d \times du).
\end{equation}
We can prove (\ref{2.13})--(\ref{2.14}) in the same way as Theorem \ref{pp1.1},
from Lemma \ref{lm3.4} and from the following:
for any $f\in C_0^{1,2}((0,1)\times \mathbb{R}^{d})$, 
\begin{eqnarray}\label{4.7}
0&=&\int_{[0,1]\times \mathbb{R}^{2d}} \left\{
\partial_tf(t, x)+
\frac{1}{\gamma}\langle\nabla_x f(t, x),u\rangle\right.\\
&&\qquad \left.
+\frac{1}{2\gamma^2}Trace (a(t, z)\nabla_x^2f(t, x))\right\}
\overline Q_\infty(dt\hbox{ }dz\hbox{ }du)\nonumber
\end{eqnarray}
(see (\ref{4.5.0})).
Indeed,  by the It\^o formula, 
\begin{eqnarray}
0&=&E\left[\int_0^1 \left\{
\partial_t f(t, \overline X_{n_k}(t))+
\frac{1}{\gamma}\langle \nabla_x f(t, \overline X_{n_k}(t)),u_{X_{n_k}(t)}\rangle\right.\right.\nonumber\\
&&\qquad\left.\left.+\frac{1}{2\gamma^2}Trace (a(t, Z_{n_k}(t))\nabla_x^2f(t, \overline X_{n_k}(t))\right\}dt\right]\nonumber
\end{eqnarray}
(see (\ref{3.19.0})).
From Lemma \ref{lm3.4}, 
\begin{eqnarray}\label{4.10}
&&E\left[\int_0^1 \left\{
\left|\partial_tf(t, \overline X_{n_k}(t)) -\partial_tf(t,X_{n_k}(t))\right|
\right.\right.\\
&&\qquad \left.\left.
+\left|\nabla_x^2f(t, \overline X_{n_k}(t))-\nabla_x^2f(t, X_{n_k}(t))\right|\right\}
dt\right]\to 0,\quad k\to\infty.\nonumber
\end{eqnarray}
From (A1, ii), for $R>0$,
\begin{eqnarray}\label{4.11}
&&E\left[\int_0^1 
|\nabla_x f(t, \overline X_{n_k}(t))-\nabla_x f(t, X_{n_k}(t))||u_{X_{n_k}(t)}|dt\right]\\
& \le& R \cdot E\left[\int_0^1 |\nabla_x f(t, \overline X_{n_k}(t))-\nabla_x f(t, X_{n_k}(t))|dt\right]\nonumber\\
&&\qquad +2\sup_{(s,x)\in [0,1]\times \mathbb{R}^d}|\nabla_x f(s,x)|\frac{1}{C_{1,R}
}
E\left[\int_0^1 L(t,Z_{n_k}(t);u_{X_{n_k}(t)})dt\right]
\to 0,\nonumber
\end{eqnarray}
as $k\to\infty$ and then $R\to\infty$ {\color{black}(see (\ref{2.1.03}) for notation)}.

In the same way as (\ref{4.7.0}), we can prove (\ref{2.15}).
{\color{black} Indeed, if the left hand side of (\ref{2.15}) is finite, then 
for sufficiently large $n\ge 1$,
take $Z_n\in\mathcal{A}^{m_n}(B_n, P_0;P_1)$ such that
\begin{equation}
V^{m_n}(B_n, P_0;P_1)\le E\left[\int_0^1 L(t,Z_{n}(t);u_{X_{n}(t)})dt\right]
\le V^{m_n}(B_n, P_0;P_1)+\frac{1}{n}.
\end{equation}
Let $\{Q_{n_k}\}_{k\ge 1}$ and  $\{P^{Z_{n_k}}\}_{k\ge 1}$
be  weakly convergent subsequences of $\{Q_n \}_{n\ge 1}$ and $\{P^{Z_n}\}_{n\ge 1}$, respectively.
Let also $\overline Q_\infty$ be the limit of $Q_{n_k} $ as $k\to\infty$ and take $X\in\mathcal{A}(P_0,P_1)$ for which (\ref{2.13})--(\ref{2.14}) hold.
Then the following holds in the same way as (\ref{4.7.0}):
\begin{eqnarray}
&&\liminf_{k\to\infty}E\left[\int_0^1 L(t,Z_{n_k}(t);u_{X_{n_k}(t)})dt\right]\label{4.13.0415}\\
& \ge&\int_{[0,1]\times \mathbb{R}^{2d}}
L(t,z;E^{\overline Q_\infty}[u|t,z])\overline Q_\infty (dt\hbox{ }dz\times \mathbb{R}^d)\nonumber\\
&=&E\left[\int_0^1 L_0(t,X(t);u_{X(t)})dt\right]
\ge V^0(P_0,P_1).\nonumber
\end{eqnarray}
}
\end{proof}

Lemma \ref{lm3.2} plays a crucial role in the proof of Theorem \ref{thm2.3}.
The key idea that leads to $(X^m, Y^m)\in\mathcal{A}^m(P_0,P_1)$ to approximate 
$X\in\mathcal{A}(P_0,P_1)$ is the following.
For $(\hat X, \hat Y)\in\mathcal{A}^m(P_0,P_1)$,
substitute $t=1$ in (\ref{3.9}). Then 
$$\hat X(1)=\hat X(0)+K^m(1) \hat Y(0)+\int_{0}^{1} K^m(1-s)(u_{\hat X}(s)ds +d W(s)),
$$
which implies that $\hat X(t)$ is close to the following semimartingale:
$$\xi (t)=\hat X(0)+K^m(1) Y(0)+\int_{0}^{t} K^m(1-s)(u_{\hat X}(s)ds +d W(s)).$$
(\ref{2.24.00}) implies that we should consider the time change of $X$ under (A2, i).

\begin{proof}[Proof of Theorem \ref{thm2.3}]

From (\ref{3.23}) and (\ref{4.13.0}), 
\begin{equation}
\quad X^m(1)
=X(\varphi^m(0))+\int_{0}^{1} K^m(1-s)(u_X^m(s)f^m(s)ds +d W^m(s))=X(1)
\end{equation}
since $\varphi^m(1)=1$ (see  (\ref{3.7})--(\ref{3.27}) for notations).

We prove that $X^m$ converges  to $X$, as $m\to 0$, locally uniformly on $[0,1)$, a.s..
From (\ref{3.26}), for $t\in [0,1)$,
\begin{eqnarray}
&&X^m(t)-X(t)\\
&=&\left(1-\frac{K^m(t)}{K^m(1)}\right)(X(0)-X(t))
-\left(\frac{1}{K^m(1)}-\frac{1}{K^m(1-t)}\right)K^m(t)X(t)\nonumber\\
&&\qquad -\frac{K^m(t)}{K^m(1-t)}(X(t)-X(\varphi^m(t)))\nonumber\\
&&\qquad 
+f^m(t)\left\{\Psi^m\left(\frac{X(\varphi^m(\cdot))}{f^m(\cdot)^2}\right)(t)
-\gamma K^m(t)\frac{X(\varphi^m(t))}{f^m(t)^2}\right\}\nonumber
\end{eqnarray}
since $\gamma K^m (1-t)=f^m(t)$ (see (\ref{2.12}) for notation).

For $\delta\in (0,1)$ and $t\in [0,1]$, considering the cases when $t\in [0,\delta]$ and when  $t\in [\delta,1]$,
\begin{eqnarray}
&&\left|\left(1-\frac{K^m(t)}{K^m(1)}\right)(X(0)-X(t))\right|\\
& \le& \sup_{0\le s\le \delta}|X(0)-X(s)|
+\frac{\exp(-\gamma\delta/m)}{1-\exp(-\gamma/m)}\times 2{\color{black}\sup_{t\in [0,1]}}|X(t)|\to 0,\nonumber
\end{eqnarray}
as $m\to 0$ and then $\delta\to 0$, since 
\begin{equation}\label{4.21}
0\le K^m(t)=\frac{1-\exp(-\gamma t/m)}{\gamma}\le K^m(1), \quad 0\le t\le 1.
\end{equation}

Take $t_0\in (0,1)$. For $t\in [0, t_0]$,
\begin{eqnarray}
&&\left|\left(\frac{1}{K^m(1)}-\frac{1}{K^m(1-t)}\right)K^m(t)X(t)\right|\\
&\le& \frac{\exp(-\gamma(1-t_0)/m)}{1-\exp(-\gamma(1-t_0)/m)}{\color{black}\sup_{t\in [0,1]}}|X(t)|
\to 0,\nonumber
\end{eqnarray}
as $m\to 0$,  from (\ref{4.21}), since $t\mapsto K^m(1-t)$ is decreasing.

From Lemma \ref{lm3.5}, 
 for $t\in [0, t_0]$,
\begin{eqnarray}
&&\left|\frac{K^m(t)}{K^m(1-t)}(X(t)-X(\varphi^m(t)))\right|\\
& \le& \frac{1}{1-\exp(-\gamma(1-t_0)/m)}\sup_{0\le s_1\le s_2\le 1, |s_1-s_2|\le 2m/\gamma}|X(s_1)-X(s_2)|\to 0, \nonumber
\end{eqnarray}
as $m\to 0$, from (\ref{4.21}), since $t\mapsto K^m(1-t)$ is decreasing.

For $t\in [0, t_0]$ and $\delta\in (0,1)$, from (\ref{3.3}) and Lemma \ref{lm3.5},
\begin{eqnarray}
&&\left|\Psi^m\left(\frac{X(\varphi^m(\cdot))}{f^m(\cdot)^2}\right)(t)
-\gamma K^m(t)\frac{X(\varphi^m(t))}{f^m(t)^2}\right|\label{3.34}\\
& \le& 2{\color{black}\sup_{t\in [0,t_0]}}\left|\frac{X(\varphi^m(t))}{f^m(t)^2}\right|
\exp\left (-\frac{\gamma\delta}{m} \right)
+\sup_{t-t\wedge\delta\le s\le t}\left|\frac{1}{f^m(s)^2}-\frac{1}{f^m(t)^2}\right|
| X(\varphi^m(s))|\nonumber\\
&&\qquad +\sup_{t-t\wedge\delta\le s\le t}\left|\frac{X(\varphi^m(s))-X(\varphi^m(t))}{f^m(t)^2}\right|\nonumber\\
& \le& 2{\color{black}\sup_{t\in [0,1]}}|X(t)|\left(\frac{1}{f^m(t_0)^2}\exp\left (-\frac{\gamma\delta}{m} \right)
+\frac{1}{f^m(t_0)^4}\exp\left(-\frac{\gamma(1-t_0)}{m}\right)\right)\nonumber\\
&&\qquad +\frac{1}{f^m(t_0)^2}\sup_{0\le s_1\le s_2\le 1, |s_1-s_2|\le \delta+2m/\gamma}
|X(s_1)-X(s_2)|\to 0,\nonumber
\end{eqnarray}
as $m\to 0$ and then $\delta\to 0$.
Indeed,  $t\mapsto f^m (t)=1-\exp (-\gamma (1-t)/m)$ is decreasing and 
$$\left|\frac{1}{f^m(s)^2}-\frac{1}{f^m(t)^2}\right|\le \frac{1}{f^m(t)^4}\times 2\exp\left(-\frac{\gamma (1-t)}{m}\right),\quad 0\le s\le t\le 1.$$

We   prove that
$Y^m$ converges  to $0$, as $m\to 0$, locally uniformly on $[0,1)$, a.s..
Take $t_0\in (0,1)$. 
Then, from (\ref{2.27}) and (\ref{3.34}), we only have to prove, to complete the proof,  that the following converges to $0$ uniformly in $t\in [0, t_0]$, as $m\to 0$:
\begin{eqnarray}\label{4.24}
&&-\frac{\exp(-\gamma t/m)}{K^m(1)}X(0)+
\frac{X(\varphi^m(t))}{K^m(1-t)}-\gamma^2K^m(t)\frac{X(\varphi^m (t))}{f^m(t)^2}\\
&=  &\frac{\exp(-\gamma t/m)}{K^m(1)}(X(\varphi^m(t))-X(0))\nonumber\\
&&\qquad +
\frac{(K^m(1-t)-K^m(t))K^m(1)-K^m(1-t)^2\exp(-\gamma t/m)}{K^m(1-t)^2K^m(1)}X(\varphi^m(t)),\nonumber
\end{eqnarray}
from (\ref{2.12}). 
We prove that (\ref{4.24}) converges to $0$ uniformly in $t\in [0, t_0]$, as $m\to 0$, a.s..
For $t\in [0, t_0]$ and $\delta\in (0,t_0)$, from Lemma \ref{lm3.5},
considering the cases when $t\in [0,\delta]$ and when  $t\in [\delta,t_0]$,
\begin{eqnarray}
&&\frac{\exp(-\gamma t/m)}{K^m(1)}|X(\varphi^m(t))-X(0)|\\
&\le&  \frac{1}{K^m(1)}\sup_{0\le s\le\min(\delta+2m/\gamma,1)} |X(s)-X(0)|
+\frac{\exp(-\gamma \delta/m)}{K^m(1)}\times 2{\color{black}\sup_{t\in [0,1]}}|X(t)|
\to 0,\nonumber
\end{eqnarray}
as $m\to 0$ and then $\delta\to 0$.
\begin{eqnarray}
&&\left|\frac{(K^m(1-t)-K^m(t))K^m(1)-K^m(1-t)^2\exp(-\gamma t/m)}{K^m(1-t)^2K^m(1)}X(\varphi^m(t))\right|\\
&\le& 
\frac{\exp(-\gamma (1-t_0)/m)}{\gamma^2K^m(1-t_0)^3}{\color{black}\sup_{t\in [0,1]}}|X(t)| \to 0,\nonumber
\end{eqnarray}
as $m\to 0$,  since
\begin{eqnarray*}
&&\gamma^2|(K^m(1-t)-K^m(t))K^m(1)-K^m(1-t)^2\exp(-\gamma t/m)|\\
&=&\exp\left(-\frac{\gamma (1-t)}{m}\right)+\exp\left(-\frac{\gamma (1+t)}{m}\right)
-2\exp\left(-\frac{\gamma }{m}\right)\le \exp\left(-\frac{\gamma (1-t)}{m}\right),
\end{eqnarray*}
and since $t\mapsto K^m(1-t)$ is decreasing. 
Here, notice that $x\mapsto \exp (-x)$ is convex.

We prove (\ref{2.29}).
From (A2, iii),
\begin{eqnarray}
&&E\left[\int_0^1L(t, Z^m(t);u_X^m(t)f^m(t))dt\right]\label{4.28}\\
& \le& 
E\left[\int_0^1R_1(t, Z^m(t); u_X(\varphi^m(t)))f^m(t)^2dt\right]
+E\left[\int_0^1L(t, Z^m(t); 0)dt\right]\nonumber\\
&=&
E\left[\int_{\varphi^m(0)}^1L((\varphi^m)^{-1}(s), Z^m((\varphi^m)^{-1}(s));u_X(s))ds\right]\nonumber\\
&&\qquad +E\left[\int_0^1(-f^m(t)^2+1)L(t, Z^m(t); 0)dt\right].\nonumber
\end{eqnarray}

From Lemma \ref{lm3.5}, for $s\in [\varphi^m(0),1]$,
\begin{eqnarray}\label{4.13}
&&L((\varphi^m)^{-1}(s), Z^m((\varphi^m)^{-1}(s));u_X(s))
\\
& \le& \Delta L\left(\frac{2m}{\gamma},\infty\right)
(1+L(s, X(s), 0; u_X(s)))+L(s, X(s), 0;u_X(s)).\nonumber
\end{eqnarray}
From (A2, ii, iv), $L(t, Z^m(t); 0)$ is bounded (see Remark \ref{rk2.1}, (iii)).
(\ref{2.29}) holds from (A2, iv).

We prove (\ref{2.30}).
From Lemmas \ref{lm3.5} and what we have proven above, 
$$((\varphi^m)^{-1}(s), X^m((\varphi^m)^{-1}(s)), Y^m((\varphi^m)^{-1}(s)))\to (s, X(s), 0), \quad m\to 0,$$
locally uniformly in $s\in (0,1)$, a.s..
$\varphi^m(0)\to 0$, as $m\to 0$.
$L(t, Z^m(t); 0)$ is bounded, and 
\begin{eqnarray*}
0&\le& -f^m(t)^2+1\le 1,  \quad 0\le t<1,\\
1-f^m(t)^2&\to& 0, \quad m\to 0, \quad 0\le t<1.
\end{eqnarray*}
(A2, ii) and (\ref{4.13}) complete the proof by Lebesgue's dominated convergence theorem.

\end{proof}

We prove Corollary \ref{thm2.2}
by showing that $P^{Y^m(0)}\in \mathcal{P}_{1,C(m,P_0,P_1)}(\mathbb{R}^d)$ for $X\in\mathcal{A}(P_0,P_1)$ for which $E\left[\int_0^1L_0(t,X(t);u_X(t))dt\right]$ is close to $V^0(P_0,P_1)$.
 
\begin{proof}[Proof of Corollary \ref{thm2.2}]
Take $X\in\mathcal{A}(P_0,P_1)$ such that 
$$E\left[\int_0^1L_0(t,X(t);u_X(t))dt\right]\le V^0(P_0,P_1)+1.$$
Then,  
\begin{equation}
E[|X(\varphi^m (0))-X(0)|]\le  K^m(1)C(m,P_0,P_1), \quad m>0,
\end{equation}
and $Z^m=(X^m, Y^m)\in \mathcal{A}^m(P_0, P_1)$ defined by  (\ref{2.21})--(\ref{2.24})
belongs to $\mathcal{A}^m(\mathcal{P}_{1,C(m,P_0,P_1)}(\mathbb{R}^d),P_0;P_1)$ for $m>0$.
Indeed, for $R>0$,
\begin{eqnarray*}
&&E[|X(\varphi^m (0))-X(0)|]\\
&\le& \frac{1}{\gamma}\left\{\int_0^{\varphi^m (0)}E[|u_X(t)|]dt+E[|W(\varphi^m (0))|]\right\}\\
&\le& \frac{1}{\gamma}\left\{\varphi^m (0) R+\frac{1}{C_{1,R}}E
\left[\int_0^1L_0(t,X(t);u_X(t))dt\right]+E[|W(\varphi^m (0))|]\right\}
\end{eqnarray*}
{\color{black}(see (\ref{2.1.03}) for notation)}.

Since $\mathcal{P}_{1,C(m,P_0,P_1)}(\mathbb{R}^d)\subset D_m$, for $m\in (0,
\varepsilon_0\gamma/2]$, from (\ref{2.29}), 
\begin{eqnarray}
V^m(D_m,P_0;P_1)&\le &V^m(\mathcal{P}_{1,C(m,P_0,P_1)}(\mathbb{R}^d),P_0;P_1)\\
&{\color{black}\le}&
{\color{black}E\left[\int_0^1L(t, Z^m(t);u_X^m(t)f^m(t))dt\right]}\nonumber\\
&\le&
\sup_{\overline m\in (0, \varepsilon_0\gamma/2]}E\left[\int_0^1L(t, Z^{\overline m}(t);u_X^{\overline m}(t)f^{\overline m}(t))dt\right]<\infty.\nonumber
\end{eqnarray}
(\ref{2.30}) completes the proof.
\end{proof}

In the proof of Theorem\ref{thm2.6},
we show that  $(\mathcal{X}^m,\mathcal{Y}^m)$ defined by (\ref{2.42})--(\ref{2.43})
is close to $(X^m, Y^m )$ defined by (\ref{2.21})--(\ref{2.24}) as $m\to 0$ and make use of 
Theorem \ref{thm2.3}.

\begin{proof}[Proof of Theorem\ref{thm2.6}]
From  (\ref{2.42})--(\ref{2.43}), (\ref{3.9}), (\ref{3.11}), (\ref{4.13.0}), and (\ref{4.19}), 
\begin{eqnarray}
\mathcal{X}^m(t)
 &=&X(0)+K^m(t)\mathcal{Y}^m(0)\label{4.35}\\
 &&\qquad +\int_0^tK^m(t-s)
\left(\left(u_{X}^m(s)
+\beta^m(0)\right)f^m(s)ds+dW^m(s)\right)\nonumber\\
&=&X^m(t)-K^m(t)\frac{X(\varphi^m(0))-X(0)-K^m(1)\mathcal{Y}^m(0)}{K^m(1)}\nonumber\\
&&\qquad +\int_0^tK^m(t-s)f^m(s)ds\cdot \beta^m(0),\nonumber\\
\mathcal{Y}^m(t)
&=&\exp\left(-\frac{\gamma t}{m}\right)\mathcal{Y}^m(0)+
\int_{0}^{t}\exp\left(-\frac{\gamma (t-s)}{m}\right)\label{4.27}\\
&&\qquad\times \left(\left(u_{X}^m(s)+\beta^m(0)\right)f^m(s)ds +dW^m(s)\right)\nonumber\\
&=&Y^m(t)-\exp\left(-\frac{\gamma t}{m}\right)\frac{X(\varphi^m(0))-X(0)-K^m(1)\mathcal{Y}^m(0)}{K^m(1)}\nonumber\\
&&\qquad +\int_{0}^{t}\exp\left(-\frac{\gamma (t-s)}{m}\right)f^m(s)ds\cdot \beta^m(0).\nonumber
\end{eqnarray}
From (\ref{2.29.00}) and (\ref{4.35}), 
\begin{equation}
\mathcal{X}^m(1)=X^m(1)=X(1),\quad \mathcal{Z}^m=(\mathcal{X}^m, \mathcal{Y}^m)\in \mathcal{A}^m(P_0,P_1)
\end{equation}
(see (\ref{2.12}), (\ref{3.27}), and (\ref{2.44}) for notations).
Indeed,
$$\int_0^1K^m(1-s)f^m(s)ds=\frac{1-\varphi^m(0)}{\gamma}.$$

From Theorem \ref{thm2.3}, Lemma \ref{lm3.5}, and (\ref{4.35})--(\ref{4.27}),
$(\mathcal{X}^m,\mathcal{Y}^m)$ converges  to $(X,0)$, as $m\to 0$, locally uniformly on $[0,1)$, a.s.
{\color{black}(see (\ref{2.44}) for notation)}.
Indeed,  as $m\to0$,
\begin{eqnarray}
\quad{\color{black}\sup_{t\in [0,1]}}|\mathcal{X}^m(t)-X^m(t)|&\le& 2(|X(\varphi^m(0))-X(0)|+
\gamma^{-1}|\mathcal{Y}^m(0)|)\to 0, \\
{\color{black}\sup_{t\in [0,1]}}|\mathcal{Y}^m(t)-Y^m(t)|
&\le& \left(\frac{1}{K^m(1)}+\frac{\gamma}{1-\varphi ^m(0)}\right)\\
&&\qquad  \times (|X(\varphi^m(0))-X(0)|+\gamma^{-1}|\mathcal{Y}^m(0)|)\nonumber\\
 &\to& 0, \nonumber
 \end{eqnarray}
since 
$$0\le K^m(t)\le K^m(1)<\frac{1}{\gamma}, \quad \gamma K^m(t-s)\le f^m(s)\le 1{\color{black},\quad 0\le s\le t\le 1}.$$

From (A2, iii),
\begin{eqnarray}\label{4.35.1}
&&E\left[\int_0^1L\left(t, \mathcal{Z}^m(t);\left(u_{X}^m(t)+\beta^m(0)\right)f^m(t)\right)dt\right]\\
& \le& 
E\left[\int_0^1R_1\left(t, \mathcal{Z}^m(t); u_X(\varphi^m(t))+\beta^m(0)\right)f^m(t)^2dt\right]
+E\left[\int_0^1L(t, \mathcal{Z}^m(t); 0)dt\right]\nonumber\\
&=&
E\left[\int_{0}^1L\left(t, \mathcal{Z}^m(t); u_X(\varphi^m(t))+\beta^m(0)\right)f^m(t)^2dt\right]\nonumber\\
&&\qquad +E\left[\int_0^1(-f^m(t)^2+1)L(t, \mathcal{Z}^m(t); 0)dt\right].\nonumber
\end{eqnarray}
{\color{black} (see (\ref{2.26.00}) for notation). From (A4),
\begin{eqnarray}
&&E\left[\int_{0}^1L\left(t, \mathcal{Z}^m(t); u_X(\varphi^m(t))+\beta^m(0)\right)f^m(t)^2dt\right]\\
&\le &E\left[\int_{0}^1L\left(t, \mathcal{Z}^m(t); u_X(\varphi^m(t))\right)f^m(t)^2dt\right]\nonumber\\
&&\qquad +C\times E\left[\int_{0}^1 |\beta^m(0)|\left(|u_X(\varphi^m(t))|^{r_0-1}
+\left|\beta^m(0)\right|^{r_0-1}\right)f^m(t)^2dt\right].\nonumber
\end{eqnarray}
The second expectation above is bounded in $m\in (0, \varepsilon_0\gamma/2]$ and converges to $0$ as $m\to 0$.
Indeed, 
}
by H\"older's inequality,
\begin{eqnarray}
&&E\left[|\beta^m(0)| \int_0^1|u_X(\varphi^m(t))|^{r_0-1}f^m(t)^2dt\right]\\
&=&E\left[|\beta^m(0)| \int_{\varphi^m(0)}^1|u_X(s)|^{r_0-1}ds\right]\nonumber\\
&\le&\{E [|\beta^m(0)|^{r_0}]\}^{1/r_0}
\left \{E\left[\int_{\varphi^m(0)}^1|u_X(s)|^{r_0}ds\right]\right\}^{(r_0-1)/r_0}.\nonumber
\end{eqnarray}
\begin{eqnarray}
&&
\left|1-\varphi ^m(0)\right|^{r_0}
E\left[\left|\beta^m(0)\right|^{r_0}\right]\\
&\le & 3^{r_0-1}
E\left[\left(\int_0^{\varphi^m(0)}|u_X(s)|ds\right)^{r_0}+|W(\varphi^m(0))|^{r_0}
+|\mathcal{Y}^m(0)|^{r_0}\right]\nonumber\\
&\le & 3^{r_0-1}
E\left[\varphi^m(0)^{r_0-1}\int_0^{\varphi^m(0)}|u_X(s)|^{r_0}ds
+\varphi^m(0)^{r_0/2}|W(1)|^{r_0}
+|\mathcal{Y}^m(0)|^{r_0}\right].\nonumber
\end{eqnarray}
For sufficiently large $R>0$,
\begin{eqnarray}
&&E\left[\int_0^{\varphi^m(0)}|u_X(s)|^{r_0}ds\right]\\
&\le & \left(\varphi^m(0)R^{r_0}+\frac{1}{C_{r_0,R}}E\left[\int_{0}^{\varphi^m(0)}L_0(t, X(t);u_X(t))dt\right]\right) \to0, \quad m\to 0,\nonumber
\end{eqnarray}
\begin{equation}\label{4.40}
E\left[\int_{\varphi^m(0)}^1|u_X(s)|^{r_0}ds\right]
\le  R^{r_0}+\frac{1}{C_{r_0, R}}E\left[\int_{0}^1L_0(t, X(t);u_X(t))dt\right]<\infty.
\end{equation}
{\color{black}(see (\ref{2.1.03}) for notation).
Replacing $Z^m$ by $\mathcal{Z}^m$ in the argument in and after (\ref{4.13}) in the proof of Theorem \ref{thm2.3},
the rest of the proof can be done exactly in the same way.}

\end{proof}

From Lemma \ref{lm3.3}, in the same way as Theorem \ref{pp1.1},
we can prove Proposition \ref{pp2.2}.

\begin{proof}[Proof of Proposition \ref{pp2.2}]
From Theorem 2.2.15 and Lemma 3.2.3 in \cite{DS}, the convex duality does the proof.
We only have to prove that $P\mapsto V^m(B,P_0;P)$ is convex and lower semicontinuous.

We first prove that $P\mapsto V^m(B,P_0;P)$ is convex.
For any  $Z_n=(X_n, Y_n)\in \mathcal{A}^m(B, P_0;P^{X_n(1)})$ such that $E[\int_0^1 |u_{X_n} (t)|dt]$ are finite, $n=0,1$,  and  for any $\lambda\in (0,1)$, let 
$$Q_\lambda(dt\hbox{ }dz\hbox{ }du):=(1-\lambda) dt P^{(Z_0(t),u_{X_0} (t))}(dz\hbox{ }du)
 +\lambda dt P^{(Z_1(t),u_{X_1} (t))}(dz\hbox{ }du).
$$
Then, in the same way as the proof of Theorem \ref{pp1.1},
 there exists $Z_\lambda \in \mathcal{A}^m$ such that 
\begin{eqnarray}
u_{X_\lambda} (t)&=&
E^{Q_\lambda}[u|t,Z_\lambda(t)],\\
P^{Z_\lambda(t)}(dz)&=&(1-\lambda) P^{Z_0(t)}(dz)+\lambda P^{Z_1(t)}(dz),
\quad 0\le t\le 1
\end{eqnarray}
by the superposition principle (see \cite{Bog21}, Theorem 1).
Here
$E^{Q_\lambda}[u|t,z]$ denotes the conditional expectation of $u$ given $(t,z)$ under $Q_\lambda$.
Indeed, replacing $P^{Z_\infty (t)}$ and $Q_\infty$ by $(1-\lambda) P^{Z_0(t)}+\lambda P^{Z_1(t)}$
and $Q_\lambda$, respectively, (\ref{4.3.0}) holds by the It\^o formula.
Since $B$ is convex, $Z_\lambda\in \mathcal{A}^m(B, P_0;(1-\lambda) P^{X_0(1)}
+\lambda  P^{X_1(1)})$.
In particular, the following holds and implies that $P\mapsto V^m(B,P_0;P)$ is convex:
\begin{eqnarray}
&&(1-\lambda) E\biggl[\int_0^1 L(t,Z_0(t);u_{X_0} (t))dt \biggr]
+\lambda E\biggl[\int_0^1 L(t,Z_1(t);u_{X_1} (t))dt \biggr]\\
&=&\int_{[0,1]\times \mathbb{R}^{2d}\times \mathbb{R}^{d} }
L(t,z;u)Q_\lambda(dt\hbox{ }dz\hbox{ }du)\nonumber\\
&\ge& \int_{[0,1]\times \mathbb{R}^{2d}}
L(t,z;E^{Q_\lambda}[u|t,z])Q_\lambda(dt\hbox{ }dz\times \mathbb{R}^{d} )
=E\biggl[\int_0^1 L(t,Z_\lambda(t);u_{X_\lambda} (t))dt \biggr],\nonumber
\end{eqnarray}
by Jensen's inequality.
From Lemma \ref{lm3.3}, in the same way as Theorem \ref{pp1.1},
one can also show the lower--semicontinuity of $P\mapsto V^m(B,P_0;P)$
{\color{black}(see also (\ref{4.13.0415}) in the proof of  Theorem \ref{thm2.1})}.
\end{proof}

For $Z=(X,Y)\in \mathcal{A}^m$,
the SDE for $X(t)+K^m(1-t)Y(t)$ plays a crucial role in the proof of Proposition \ref{pp2.3}.

\begin{proof}[Proof of Proposition \ref{pp2.3}]
For $Z=(X,Y)\in \mathcal{A}^m(P_0,P^{X(1)})$, let
\begin{equation}\label{2.34}
\eta(t):=X(t)+K^m(1-t)Y(t)
\end{equation}
(see (\ref{3.7}) for notation).
Then
\begin{eqnarray}
d\eta (t)&=&K^m(1-t)(u_X(t)dt+dW(t)),\label{2.35}\\
\eta(1)&=&X(1).\nonumber
\end{eqnarray}
Indeed, by the It\^o formula, for $t\in (0,1)$,
\begin{eqnarray*}
d\{X(t)+K^m(1-t)Y(t)\}
&=&\frac{1}{m}Y(t)dt-\frac{1}{m}\exp\left(-\frac{\gamma}{m}(1-t)\right)Y(t)dt\\
&&\qquad +K^m(1-t)\left\{\left(u_X(t)-\frac{\gamma }{m}Y(t)\right)dt+dW(t)\right\}.
\end{eqnarray*}
From (\ref{2.35}), $u_\eta =u_X$, and 
\begin{equation}
E\biggl[f(X (1))-\int_0^1 L(t;u_X(t))dt \biggr]
=E\biggl[f(\eta (1))-\int_0^1 L(t;u_\eta(t))dt \biggr],
\end{equation}
which implies that the l.h.s. is less than or equal to the r.h.s. in (\ref{2.34.1}).

Suppose that $\{\eta (t)\}_{0\le t\le 1}$ be a semimartingale defined on a complete filtered probability space such that the following holds:
\begin{equation}
d\eta (t)=K^m(1-t)\left\{u_\eta(t)dt+dW(t)\right\}.
\end{equation}
Let $X_0$  be an $\mathbb{R}^d$--valued random variable defined
on the same probability space as  $\{\eta (t)\}_{0\le t\le 1}$ such that $P^{X_0}=P_0$.
Let $Z=(X,Y)$ be a solution to (\ref{1.1})--(\ref{1.2}) such that 
$$
u_X(t):=u_\eta(t),\quad X(0)=X_0, \quad Y(0)=\frac{\eta(0)-X_0}{K^m(1)}.$$
Then
\begin{eqnarray}
d\eta (t)&=&K^m(1-t)\left\{u_X(t)dt+dW(t)\right\},\label{4.37}\\
\eta(1)&=&X(1).\nonumber
\end{eqnarray}
Indeed, from (\ref{2.34})--(\ref{2.35}), 
\begin{eqnarray*}
X(1)&=&X(1)+K^m(1-1)Y(1)\\
&=&X(0)+K^m(1-0)Y(0)
+\int_0^1 d(X(t)+K^m(1-t)Y(t))\\
&=&\eta(0)+\int_0^1 K^m(1-t)\left\{u_X(t)dt+dW(t)\right\}\\
&=&\eta(1).
\end{eqnarray*}
From (\ref{4.37}),
\begin{equation}
E\biggl[f(\eta (1))-\int_0^1 L(t;u_\eta(t))dt \biggr]
=E\biggl[f(X (1))-\int_0^1 L(t;u_X(t))dt \biggr].
\end{equation}
Since $P^{X(0)}=P^{X_0}=P_0$,
the l.h.s. is greater than or equal to the r.h.s. in  (\ref{2.34.1}).

If $L=|u|^2/2$, then $V^0(P_0,P_1)$ is Schr\"odinger's problem  (see section \ref{sec:1}).
If $V^m(P_0,P_1)$  converges to $V^0(P_0,P_1)$ as $m\to 0$,
then $V^0(P_0,P_1)$ does not depend on $P_0$ either, which is a contradiction
 (see \cite{M2021} for more general SOTs).
 
\end{proof}

\noindent
{\bf Acknowledgements.}
The author would like to thank the anonymous referees for their constructive comments that improved the
quality of this paper.


\end{document}